\documentclass[11pt]{article}

\usepackage[margin=1.05in]{geometry}
\usepackage{amsmath,amssymb,amsthm,mathtools}
\usepackage{graphicx}
\usepackage{microtype}
\usepackage[numbers,sort&compress]{natbib}
\usepackage{xcolor}
\usepackage[colorlinks=true,linkcolor=blue!60!black,citecolor=blue!60!black,urlcolor=blue!60!black]{hyperref}
\usepackage[capitalise,noabbrev,nameinlink]{cleveref}
\crefname{assumption}{Assumption}{Assumptions}
\crefname{observation}{Observation}{Observations}

\usepackage{aliascnt}
\usepackage{enumitem}
\usepackage{booktabs}

\newtheorem{theorem}{Theorem}[section]

\newaliascnt{proposition}{theorem}
\newtheorem{proposition}[proposition]{Proposition}
\aliascntresetthe{proposition}

\newaliascnt{corollary}{theorem}
\newtheorem{corollary}[corollary]{Corollary}
\aliascntresetthe{corollary}

\newaliascnt{lemma}{theorem}
\newtheorem{lemma}[lemma]{Lemma}
\aliascntresetthe{lemma}

\theoremstyle{definition}
\newaliascnt{definition}{theorem}

\aliascntresetthe{definition}

\newtheorem{assumption}{Assumption}

\newaliascnt{example}{theorem}
\newtheorem{example}[example]{Example}
\aliascntresetthe{example}

\theoremstyle{remark}
\newaliascnt{remark}{theorem}
\newtheorem{remark}[remark]{Remark}
\aliascntresetthe{remark}

\newaliascnt{observation}{theorem}
\newtheorem{observation}[observation]{Observation}
\aliascntresetthe{observation}

\newcommand{\R}{\mathbb{R}}
\newcommand{\Z}{\mathbb{Z}}
\newcommand{\E}{\mathbb{E}}
\newcommand{\Prob}{\mathbb{P}}
\newcommand{\Var}{\operatorname{Var}}
\newcommand{\sgn}{\operatorname{sgn}}
\newcommand{\artanh}{\operatorname{artanh}}
\newcommand{\PGW}{\mathrm{PGW}}
\newcommand{\CSBM}{\mathrm{CSBM}}
\newcommand{\err}{\mathcal{E}}
\newcommand{\Poi}{\mathrm{Poi}}
\newcommand{\dd}{\,\mathrm{d}}

\title{\bf The Value of Depth in Message Passing\\ on Sparse Graphs: A Kesten--Stigum Dichotomy}
\author{Aseem Baranwal\\
\small David R.\ Cheriton School of Computer Science, University of Waterloo\\
\small \texttt{aseemrb@gmail.com}}
\date{}

\begin{document}
\maketitle

\begin{abstract}
How deep does a graph neural network need to be on a sparse graph? We study
this question in its purest statistical form: node classification on the
sparse contextual stochastic block model ($\CSBM$) with average degree
$\Delta=O(1)$, whose local weak limit is a broadcast-labelled Poisson
Galton--Watson tree. On this limit, prior work derived a message-passing
classifier $h_\ell$ that aggregates, from every vertex $v$ at distance $k\le
\ell$ from the root, the distance-attenuated evidence $2\artanh\!\big(\gamma^k
t(X_v)\big)$, where $\gamma$ is the edge signal and $t$ is a bounded
likelihood-ratio transform of the feature. We prove that the value of depth
for this classifier is governed by a single number, the Kesten--Stigum ratio
$\kappa=\gamma^2\Delta$. Below the Kesten--Stigum threshold ($\kappa<1$),
the error sequence $(\err(\ell))_{\ell\ge0}$ is Cauchy at a geometric rate:
$|\err(\ell)-\err(\ell')|\le C\kappa^{(\ell+1)/3}$ for all $\ell'>\ell$, so
all layers beyond depth $O(\log(1/\epsilon))$ change the error by less than
$\epsilon$; conversely, under mild regularity every sufficiently deep layer still
\emph{changes the decision} with probability at least $c\,\kappa^{\ell/2}$,
matching the empirically sharp saturation exponent. Above the threshold
($\kappa>1$), depth is geometrically
productive: $\err(\ell)$ is driven to a floor at most of order
$1/(\kappa-1)$, set by the branching process, at a geometric rate that can
be taken $\kappa^{-s\ell}$ for any $s<1$ (the floor constant is
unoptimized, and this bound has content only for $\kappa>17$). In both
regimes, no
local classifier of any depth beats the universal floor
$e^{-\Delta}\,\Phi(-\zeta)$ set by isolated roots, where $\zeta$ is the
feature signal-to-noise ratio; conversely, the first layer provably helps,
by an explicit total-variation amount. The proofs use two-type
branching-process martingales, correlation decay of broadcast processes, an
antisymmetry identity for the messages, and anti-concentration. We
complement the theory with simulations, including an exact
belief-propagation baseline on the same trees. The error
curve of $h_\ell$ is mildly non-monotone in $\ell$---the pairwise rule
treats correlated deep evidence as independent, so an optimal finite depth
exists (an exact instance is certified in the appendix)---while exact BP saturates strictly faster than the linearized
rule, at an effective per-layer ratio below $\kappa$ that we identify.
\end{abstract}

\section{Introduction}\label{sec:intro}

Graph neural networks (GNNs) aggregate information over neighbourhoods of
increasing radius as layers are stacked, and the choice of depth is among the
most basic of their design decisions. Empirically, deep graph convolutional
networks \citep{kipf2017semi} often perform \emph{worse} than shallow ones,
a phenomenon usually
attributed to oversmoothing~\citep{li2018deeper,oono2020graph,keriven2022not,
wu2023nonasymptotic} or to bottlenecks in information flow
\citep{alon2021bottleneck}. These explanations concern particular
architectures and their training dynamics. A more basic statistical question
sits underneath them, and it can be posed exactly:
\begin{quote}
\emph{On a sparse graph, how much statistical value does the $\ell$-th layer
of message passing add to the classification of a node?}
\end{quote}

This paper answers the question for node classification on the sparse
contextual stochastic block model \citep[$\CSBM$;][]{deshpande2018contextual},
the canonical setting where each vertex carries both a latent class label and
a feature vector, edges within classes appear with probability $a/n$ and
across classes with probability $b/n$, and the average degree
$\Delta=(a+b)/2$ stays bounded as the graph grows. In this regime the
$\ell$-neighbourhood of a typical vertex converges, in the local weak sense
\citep{benjamini2001recurrence,aldous2004objective}, to a Poisson
Galton--Watson tree carrying a broadcast process: each child keeps its
parent's label with probability $(1+\gamma)/2$, where
$\gamma=(a-b)/(a+b)$ is the edge signal, and every vertex independently draws
a feature from a class-conditional distribution. Local classification on the
sparse $\CSBM$ is thus classification of the root of a noisy broadcast tree
with side observations---the setting of a long and beautiful line of work on
reconstruction on trees
\citep{kesten1966limit,evans2000broadcasting,mossel2001reconstruction,
mossel2003information,janson2004robust} and community detection
\citep{decelle2011asymptotic,massoulie2014community,mossel2015reconstruction,
mossel2018proof,abbe2018community}.

For this limit, \citet{baranwal2023optimality} (see also
\citep{baranwal2024thesis}) derived a message-passing classifier that is
realizable by a GNN and computed its exact generalization error. The
classifier at depth $\ell$ forms the statistic
\begin{equation}\label{eq:intro-statistic}
T_\ell \;=\; \sum_{k=0}^{\ell}\;\sum_{v\,:\,d(o,v)=k}
2\artanh\!\big(\gamma^{k}\, t(X_v)\big),
\qquad
t(x) \;=\; \frac{\rho_{+}(x)-\rho_{-}(x)}{\rho_{+}(x)+\rho_{-}(x)},
\end{equation}
and predicts $\sgn(T_\ell)$ for the root $o$; here $\rho_{\pm}$ are the
class-conditional feature densities. Two structural features of
\eqref{eq:intro-statistic} matter for us. First, the evidence of a vertex at
distance $k$ is passed through the \emph{contraction} $z\mapsto
2\artanh(\gamma^k z)$: messages from depth $k$ are automatically clipped to
$[-2\artanh(\gamma^k),\,2\artanh(\gamma^k)]$, a shrinking window. Depth
attenuation is thus built into the statistically-derived rule, in contrast
with vanilla graph convolutions, which weight all hops of a deep stack
uniformly. Second, the rule is \emph{pairwise}: each vertex contributes
through its distance from the root alone, which makes the statistic a sum of
per-vertex terms and renders a sharp depth analysis possible. The scope
is deliberately statistical: we fix this rule, with its
statistically-derived weights, and use exact belief propagation as the
Bayes baseline---no training is involved. What the results say about
\emph{learned} GNNs is a question of mechanism transfer, which we take up
in \cref{sec:related,sec:discussion}.

Our contribution is a quantitative picture of the map
$\ell\mapsto\err(\ell)$, where $\err(\ell)$ is the misclassification
probability of $h_\ell=\sgn(T_\ell)$ on the limit tree. The answer is a
dichotomy organized by the \emph{Kesten--Stigum ratio}
\begin{equation}\label{eq:ks-ratio}
\kappa \;=\; \gamma^2\Delta \;=\; \frac{(a-b)^2}{2(a+b)},
\end{equation}
the quantity that separates the reconstructible from the non-reconstructible
regime for the binary broadcast process
\citep{kesten1966limit,evans2000broadcasting}
and for the sparse stochastic block model
\citep{decelle2011asymptotic,massoulie2014community,mossel2015reconstruction,
mossel2018proof}.

\subsection{Contributions}\label{sec:contributions}

\paragraph{Below the threshold, depth saturates geometrically
(\cref{thm:below}).} If $\kappa<1$ then for all $\ell'>\ell\ge 0$,
\[
|\err(\ell)-\err(\ell')| \;\le\; C\,\kappa^{(\ell+1)/3},
\]
with an explicit constant $C$ depending only on the model. Consequently
$\err(\infty)=\lim_{\ell\to\infty}\err(\ell)$ exists and depth
$\ell(\epsilon)= \lceil 3\log(C/\epsilon)/\log(1/\kappa)\rceil$ already
places the classifier within $\epsilon$ of everything deeper. The mechanism
is a cancellation revealed by an antisymmetry identity for the messages: the
\emph{expected} evidence carried by generation $k$ scales as $\kappa^k$, not
as the naive count $(\gamma\Delta)^k$ of signed vertices, and below the
threshold its second moment vanishes geometrically. A companion lower
bound (\cref{thm:fliplower}) shows the saturation is genuinely
geometric and pins its exponent from below: under mild regularity of the
feature law, every layer beyond an explicit initial depth flips the
decision with probability at
least $c\,\kappa^{\ell/2}$---the exponent the simulations identify as
sharp.

\paragraph{Above the threshold, depth amplifies geometrically
(\cref{thm:above}).} If $\kappa>1$ then for every $\ell\ge 1$ and
every $s\in(0,1)$,
\[
\err(\ell)\;\le\; \frac{16}{\kappa-1} \;+\; 2\,\kappa^{-s\ell}
\;+\;\exp\!\Big(\!-c_0\,(\kappa-1)\,\frac{\kappa^{(1-s)\ell-1}}{\ell^2}\Big)
\;+\;\exp\!\Big(\!-\tfrac{\vartheta}{2}\,\kappa^{\ell}\Big),
\]
where $c_0=\vartheta^2/(8c_\gamma^2)$ and $\vartheta=\E_+[t(X)^2]$ is the
feature information of \cref{ass:informative}. Each additional
layer multiplies the signal-to-noise ratio of the decision statistic by
$\kappa$, and the error is driven to a floor governed by the fluctuations
and extinction behaviour of the underlying branching process. The
genuinely signal-driven error terms decay super-geometrically; the
geometric term $\kappa^{-s\ell}$ is a truncation cost with $s$ a free
parameter, so the approach to the floor is faster than every geometric
rate $\kappa^{-s\ell}$ with $s<1$. We flag at the outset that the floor
constant is unoptimized and the bound is non-vacuous only for $\kappa>17$
(hence average degree above $17$); in the sparse configurations we
simulate, the above-threshold half of the dichotomy is supported by the
experiments, not certified by the theorem
(\cref{sec:experiments}).

\paragraph{A universal floor (\cref{prop:floor}).} For every depth
$\ell$ and \emph{every} $\ell$-local classifier whatsoever,
\[
\Prob(\text{error}) \;\ge\; e^{-\Delta}\,\varepsilon_{\mathrm{feat}},
\]
where $\varepsilon_{\mathrm{feat}}$ is the Bayes error of the feature mixture
($\Phi(-\zeta)$ for the Gaussian mixture with signal-to-noise ratio $\zeta$).
Depth is never a substitute for feature signal on a sparse graph: a
constant fraction of vertices are isolated and must be classified from their
features alone. Conversely, the first layer provably helps
(\cref{prop:first}): $\err(0)-\err(1)$ is bounded below by an
explicit positive quantity built from the total-variation distance of the
depth-$1$ observation and the mass the root's log-likelihood ratio puts
near zero---a certified strict improvement, and the only \emph{lower}
bound on the value of a layer we prove (deeper layers are bounded above
only; see \cref{sec:discussion}).

\paragraph{Depth prescriptions (\cref{cor:depth}).}
In both regimes, depth
$O(\log(1/\epsilon))$ \emph{suffices} for $\epsilon$-accuracy relative to
the relevant limit, independent of the graph size: below the
threshold because deeper layers are provably (nearly) worthless, above it
because the error reaches its floor within $O(\log_\kappa(1/\epsilon))$
layers. We emphasize that the prescription is $O(\cdot)$ and not
$\Theta(\cdot)$: apart from the first
layer (\cref{prop:first}), we do not prove that fewer than
logarithmically many layers are insufficient. The simulations are
consistent with a matching
$\kappa^{\ell/2}$ lower bound below the threshold, which we leave open
(\cref{sec:discussion}). On a finite graph of $n$ vertices, local
information extends to radius $O(\log n)$; the useful depth is a constant,
far below that horizon.

\paragraph{Two structural remarks.} First, we observe
(\cref{rem:bp}) that for $\ell\ge2$ the pairwise rule
\eqref{eq:intro-statistic} is the \emph{linearization} of exact belief
propagation on the observed tree: the two coincide on stars and on
single-source paths, and differ as soon as feature evidence stacks along a
path, because the pairwise rule attenuates each source separately while BP
attenuates the combined subtree evidence. Second, $\err(\ell)$ is mildly
\emph{non-monotone}: beyond a finite optimal depth, additional layers
\emph{hurt} the pairwise rule, because it treats deep messages that share
ancestry---hence are positively correlated---as independent
(\cref{obs:nonmono}); an exact instance---a two-point
feature law with $\err(3)-\err(2)\ge0.037$---is certified in
\cref{app:certified}. Exact BP cannot
degrade with depth (its depth-$\ell$ posterior is Bayes for the depth-$\ell$
observation), so this is a genuine cost of linearization; our saturation
theorem caps the damage below the threshold. The comparison also
uncovers a rate-level asymmetry: exact BP's decisions settle at a
strictly faster geometric rate than the pairwise rule's, governed by an
effective ratio $\kappa_{\mathrm{BP}}<\kappa$ that we identify and verify
by population dynamics (\cref{sec:experiments}).

\subsection{Related work}\label{sec:related}

\paragraph{Reconstruction on trees and the Kesten--Stigum threshold.} The
ratio \eqref{eq:ks-ratio} appears first in the study of multi-type branching
processes \citep{kesten1966limit}: the signed population martingale is
$L^2$-bounded if and only if $\kappa>1$. For broadcast processes, $\kappa>1$
implies reconstructability of the root from deep leaves
\citep{evans2000broadcasting,mossel2003information}, and the \emph{robust}
reconstruction problem---root recovery from noisy observations of deep
generations---is determined by $\kappa$ exactly \citep{janson2004robust}. Our
below-threshold theorem is a quantitative, feature-decorated relative of
robust reconstruction impossibility: it bounds the total influence of all
generations beyond $\ell$ on the \emph{decision}, at an explicit geometric
rate, for the specific statistic used by the message-passing classifier. Our
above-threshold theorem is an effective version of reconstructability, again
with rates, for the same statistic.

\paragraph{Sparse SBM and CSBM.} For the sparse stochastic block model
without features, $\kappa>1$ is the threshold for detecting communities
\citep{decelle2011asymptotic,massoulie2014community,mossel2015reconstruction,
mossel2018proof}; see \citet{abbe2018community} for a survey. The contextual
model couples the graph with a Gaussian (or general) mixture on features
\citep{deshpande2018contextual}. Labelled/side-information variants of the
block model were studied by \citet{kanade2016global}, who showed that
vanishing amounts of label information interact subtly with the threshold.
The classifier we analyze is the message-passing architecture of
\citet{baranwal2023optimality}, which computes, for each vertex of the local
neighbourhood, evidence attenuated through the $k$-step class-transition
matrix, together with its exact generalization error at fixed depth and its
realizability by a GNN with $O(\ell)$ message-passing rounds; the thesis
\citep[Ch.~5]{baranwal2024thesis} gives an expanded exposition. Our
contribution is orthogonal: we fix the rule and quantify the \emph{marginal
value of depth}. A caution on regimes: the sharp \emph{detection}
thresholds for the CSBM, conjectured by \citet{deshpande2018contextual} and
proved by \citet{lu2023contextual}, live in the high-dimensional
proportional regime ($d/n\to$ const, per-vertex feature signal-to-noise
vanishing), where features \emph{shift} the location of the threshold. Our
regime is different---$d$ is fixed and the per-vertex feature
signal-to-noise ratio is a constant---so the root is always non-trivially
classifiable from its own feature, detection is never at issue, and
$\kappa$ governs only the \emph{value of depth}. This is why the
graph-only Kesten--Stigum ratio appears here with no feature correction.

\paragraph{Linearized BP and the non-backtracking operator.} The
identification of the pairwise rule as \emph{linearized} belief
propagation (\cref{rem:bp}) has a substantial lineage in community
detection: linearizing BP around its uninformative fixed point yields the
non-backtracking operator of \citet{krzakala2013spectral}, whose spectrum
detects communities down to the Kesten--Stigum threshold
\citep{bordenave2018nonbacktracking}, and the achievability of the
threshold in general block models was established through a linearized
\emph{acyclic} BP \citep{abbe2016achieving}. The pairwise rule is the
natural vertex-wise expression of the same linearization on a tree: each
vertex's evidence reaches the root along its unique (automatically
non-backtracking) path, attenuated by a factor $\gamma$ per edge---the
$\gamma^k$ in \eqref{eq:intro-statistic}---while exact BP attenuates
whole subtree \emph{aggregates} through the same per-edge contraction.
Those works use the linearization as an \emph{algorithm} near the
threshold, where BP's fixed point is uninformative; we quantify, with
feature side information, what the linearized rule gains or forfeits
\emph{per layer} at any fixed $\kappa$. The closest \emph{learned}
counterpart is the line-graph network of \citet{chen2019supervised},
which trains a GNN built on the non-backtracking operator to the
Kesten--Stigum threshold on the SBM; our fixed-weight analysis can be
read as the statistical skeleton such architectures approximate.

\paragraph{Effects of graph convolutions in denser regimes.} When the
average degree grows like $\log^2 n$ or faster, graph convolutions improve
linear separability thresholds by degree-dependent factors
\citep{baranwal2021graph,baranwal2023effects}, graph attention
\citep{velickovic2018graph} has provable limitations in hard regimes
\citep{fountoulakis2023graph}, and repeated
\emph{corrected} convolutions (with the principal component removed) enjoy
exponential improvements per round \citep{wang2024analysis}. These analyses
rely on degree concentration, which fails precisely in the sparse regime we
study; conversely, the tree-limit machinery used here has no dense analogue.
The two regimes give complementary answers to the depth question: in dense
graphs a couple of convolutions already mix the whole graph, while in sparse
graphs depth is information-limited and the Kesten--Stigum ratio decides
whether the information is there at all.

\paragraph{Oversmoothing and depth in GNNs.} The empirical failure of deep
GCNs motivated a line of theory on oversmoothing
\citep{li2018deeper,oono2020graph,keriven2022not,wu2023nonasymptotic},
oversquashing \citep{alon2021bottleneck}, and depth--width expressivity
limits \citep{loukas2020graph}. Our results offer a
complementary, purely statistical account for sparse graphs: the
statistically-derived rule attenuates depth-$k$ evidence by $\gamma^k$ and
consequently never \emph{collapses} with depth the way uniform aggregation
does---below the threshold its error saturates, with at most a mild upward
drift that \cref{thm:below} caps
(\cref{obs:nonmono}), and above the threshold it improves to a
floor---while uniform aggregation corresponds to ignoring the
attenuation and is mismatched to the broadcast structure. The mild
non-monotonicity we observe is a distinct, finer effect: a mismatch between
the pairwise independence assumption and ancestral correlations, invisible at
the level of oversmoothing analyses.

\section{Model and preliminaries}\label{sec:model}

\subsection{The sparse CSBM and its local limit}\label{sec:csbm}

For $n\in\mathbb{N}$ and constants $a>b>0$, the binary symmetric sparse
contextual stochastic block model $\CSBM(n,\mathbb{P}_{\pm},a/n,b/n)$
generates: labels $y_1,\dots,y_n\in\{\pm1\}$ i.i.d.\ uniform; an undirected
graph $G_n$ on $[n]$ where each pair $u\ne v$ is an edge independently with
probability $a/n$ if $y_u=y_v$ and $b/n$ otherwise; and features $X_u\in\R^d$
drawn independently with $X_u\sim \mathbb{P}_{y_u}$. We write
\[
\Delta=\frac{a+b}{2}\,,\qquad \gamma=\frac{a-b}{a+b}\in(0,1)\,,\qquad
\kappa=\gamma^2\Delta=\frac{(a-b)^2}{2(a+b)}\,,
\]
and assume throughout that $\Delta>1$. Note that $\kappa<\Delta$ always
(since $\gamma<1$): large values of $\kappa$ demand at least as large an
average degree. The heterophilous case $b>a$ is
symmetric under $\gamma\mapsto-\gamma$ and is discussed in
\cref{sec:discussion}.

As $n\to\infty$ with $(a,b,\mathbb{P}_\pm)$ fixed, the feature-decorated
rooted graph $(G_n,o)$ with a uniformly random root $o$ converges locally
weakly \citep{benjamini2001recurrence,aldous2004objective,
mossel2015reconstruction,deshpande2018contextual} to the following limit
object $(T,o,y,X)$:
\begin{enumerate}[leftmargin=2em,itemsep=0.1em]
\item the tree $T$ is a Poisson Galton--Watson tree $\PGW(\Delta)$ rooted at
$o$: every vertex independently has $\Poi(\Delta)$ children;
\item labels follow a \emph{broadcast process}: $y_o$ is uniform on
$\{\pm1\}$, and each child independently keeps its parent's label with
probability $\tfrac{1+\gamma}{2}$ and flips it with probability
$\tfrac{1-\gamma}{2}$;
\item features are conditionally independent given the labels, with
$X_v\sim\mathbb{P}_{y_v}$.
\end{enumerate}
Equivalently (by Poisson thinning) each vertex has $\Poi(a/2)$ children of
its own label and, independently, $\Poi(b/2)$ children of the opposite label;
in particular the shape of $T$ is independent of the labels. We use the
following notation: $N_k=N_k(o)$ is the set of vertices at graph distance
exactly $k$ from the root, $S_k=|N_k|$, and for $v\in T$ we write
$\sigma_v=y_v\,y_o\in\{\pm1\}$ for its label relative to the root. All
probabilities and expectations below are on this limit object; the transfer
of our conclusions back to finite graphs is discussed in
\cref{rem:finite}.

\subsection{The pairwise message-passing classifier}\label{sec:classifier}

Let $\rho_\pm$ denote densities of $\mathbb{P}_{\pm}$ with respect to a
common base measure $\mu$ (\cref{ass:sym} below ensures these
exist and are mutually absolutely continuous), let
$\psi=\rho_+/\rho_-$ be the likelihood ratio, and define the bounded
transform
\begin{equation}\label{eq:t-def}
t(x)\;=\;\frac{\rho_+(x)-\rho_-(x)}{\rho_+(x)+\rho_-(x)}\;\in\;[-1,1],
\qquad \log\psi(x) = 2\artanh(t(x)).
\end{equation}
For depth $\ell\ge0$, the \emph{pairwise message-passing classifier} is
\begin{equation}\label{eq:classifier}
h_\ell \;=\; \sgn(T_\ell),\qquad
T_\ell \;=\; \sum_{k=0}^{\ell}\sum_{v\in N_k} M_k(X_v),\qquad
M_k(x) \;=\; 2\artanh\!\big(\gamma^k\,t(x)\big).
\end{equation}
Note $M_0(x)=\log\psi(x)$, so at $\ell=0$ the rule is the feature-only Bayes
classifier. The rule \eqref{eq:classifier} is precisely the classifier
derived by \citet{baranwal2023optimality} (Theorem~1 and Corollary~1.1
there\footnote{Throughout, theorem numbers and the form of the classifier
refer to the updated arXiv version of \citet{baranwal2023optimality}
(arXiv:2305.10391), which states the rule in the marginalized form used
here: each vertex contributes the likelihood of its own feature
\emph{marginalized} over its label given the root's, replacing the joint
maximization over neighbourhood labels used in the proceedings version;
see also \citet[Ch.~5]{baranwal2024thesis}.}): there, the message
contributed by a vertex $v$ at distance $k$ is
$\log\langle \rho(X_v),Q^k_{y_o}\rangle$-type evidence attenuated through the
$k$-step class-transition matrix, which in the binary symmetric case
simplifies to exactly $M_k(X_v)$ using
$\tfrac{p_k}{q_k}=\tfrac{1+\gamma^k}{1-\gamma^k}$. It is implementable by a
message-passing GNN with $\ell$ rounds, and for $\ell=1$ it coincides with
the exact Bayes rule for the depth-$1$ observation. We define the error
\begin{equation}\label{eq:error}
\err(\ell) \;=\; \Prob\big(y_o\,T_\ell \le 0\big),
\end{equation}
counting ties conservatively as errors (under \cref{ass:density}
ties are null events).

Two elementary properties of the messages drive everything that follows.
First, they are \emph{clipped}: for $k\ge1$, $|M_k(x)|\le c_k :=
2\artanh(\gamma^k)$, and
$c_k\le c_\gamma\gamma^k$ with $c_\gamma := 2/(1-\gamma^2)$, so a vertex at
depth $k\ge1$ can move the statistic by at most $O(\gamma^k)$ no matter how
loud its feature (only the root's own evidence $M_0=\log\psi$ is
unbounded). Second, they are \emph{antisymmetric} in the class
(\cref{lem:messages}): the conditional means $\pm m_k$ under the two
classes cancel to first order when summed over a generation whose signed
excess is small, and $m_k\asymp\gamma^k$. Since generation $k$ contains
$\approx\Delta^k$ vertices with signed excess $\approx(\gamma\Delta)^k$, the
generation contributes signal $\approx m_k(\gamma\Delta)^k\asymp\kappa^k$ and
noise of standard deviation $\approx\gamma^k\sqrt{\Delta^k}=\kappa^{k/2}$:
the Kesten--Stigum ratio emerges as the per-layer signal gain.

\subsection{Assumptions}\label{sec:assumptions}

\begin{assumption}[symmetric, mutually absolutely continuous features]
\label{ass:sym}
$\mathbb{P}_+$ and $\mathbb{P}_-$ are mutually absolutely continuous, with
densities $\rho_\pm$ with respect to a base measure $\mu$, and there is a
$\mu$-measure-preserving involution $\tau:\R^d\to\R^d$ with
$\rho_-(x)=\rho_+(\tau(x))$ for $\mu$-a.e.\ $x$.
\end{assumption}

\begin{assumption}[informative features]\label{ass:informative}
$\vartheta := \E_+\big[t(X)^2\big]>0$, where $\E_+$ denotes expectation under $X\sim\mathbb{P}_+$.
\end{assumption}

\begin{assumption}[non-atomic log-likelihood ratio; only for
\cref{thm:below}]\label{ass:density}
Under $X\sim\mathbb{P}_+$, the random variable $\log\psi(X)$ has a density on
$\R$ bounded by $B<\infty$.
\end{assumption}

\begin{example}[Gaussian mixture]\label{ex:gaussian}
Let $\mathbb{P}_\pm=\mathcal{N}(\pm\boldsymbol{\mu},\sigma^2 I_d)$ with
$\boldsymbol{\mu}\neq\mathbf{0}$, and set
$\zeta=\|\boldsymbol{\mu}\|_2/\sigma>0$. Then \cref{ass:sym} holds
with $\tau(x)=-x$;
$\log\psi(X)\,|\,y=+1\sim\mathcal{N}(2\zeta^2,4\zeta^2)$, so
\cref{ass:density} holds with $B=(2\zeta\sqrt{2\pi})^{-1}$; and
$\vartheta=\E\tanh^2(Z/2)>0$ for $Z\sim\mathcal{N}(2\zeta^2,4\zeta^2)$, so
\cref{ass:informative} holds. The feature-only Bayes error is
$\err(0)=\Phi(-\zeta)$.
\end{example}

\cref{ass:sym} encodes the symmetry of the two-class model and is
standard in this line of work (the Gaussian mixture with opposite means being
the canonical case). Under it, $t(\tau(x))=-t(x)$ and hence
$M_k(\tau(x))=-M_k(x)$: swapping the classes negates every message. This
innocuous-looking identity is the engine of the below-threshold
cancellation---and the mechanism is genuinely symmetry-bound: under class
asymmetry the mean message of generation $k$ acquires a label-independent
drift that $\sgn(T_\ell)$ does not re-center, and the natural object
becomes a recentred statistic (see \cref{sec:discussion}).
\cref{ass:density} enters exactly once, through the
anti-concentration \cref{lem:anticonc}; it excludes atomic features
(e.g.\ discrete node attributes) from \cref{thm:below} as stated.
The natural relaxation is to replace $2Bs$ there by the L\'evy
concentration function of $T_{\ell'}$ at scale $s$, which is what the
proof actually uses; we do not pursue the bookkeeping of ties here.

\section{Main results}\label{sec:results}

Throughout this section, \crefrange{ass:sym}{ass:informative}
are in force, and $c_\gamma=2/(1-\gamma^2)$,
$\vartheta=\E_+[t(X)^2]$.

\subsection{Below the threshold: geometric saturation}

\begin{theorem}[depth saturation, $\kappa<1$]\label{thm:below}
Let $\kappa<1$ and let \cref{ass:density} hold with bound $B$. Set
\[
C_1 \;=\; \frac{3\,c_\gamma^2\,\vartheta}{(1-\kappa)^3}.
\]
Then for all $0\le\ell<\ell'<\infty$,
\begin{equation}\label{eq:below-main}
\big|\err(\ell)-\err(\ell')\big| \;\le\;
3\,B^{2/3}\,C_1^{1/3}\;\kappa^{(\ell+1)/3}.
\end{equation}
In particular, $\err(\infty):=\lim_{\ell'\to\infty}\err(\ell')$ exists and
$|\err(\ell)-\err(\infty)|\le 3B^{2/3}C_1^{1/3}\kappa^{(\ell+1)/3}$ for
every $\ell\ge 0$. Moreover $T_\ell$ converges in $L^2$ to a limit
statistic $T_\infty$, and $\err(\infty)=\Prob(y_o T_\infty\le0)$: the
limit is the error of a bona fide infinite-depth classifier, not only a
limit of numbers.
\end{theorem}

The theorem controls the total contribution of all layers beyond $\ell$
in one statement (a per-increment geometric bound would recover this by
telescoping, at the price of a further $(1-\kappa^{1/3})^{-1}$ in the
constant, so the uniformity is a convenience of the formulation rather
than extra strength). Its content is best appreciated against the naive
first-moment bound: generation $k$ contains $\Delta^k$ vertices in
expectation, each contributing a message of size up to
$c_\gamma\gamma^k$, so the crude expected total perturbation from depths
beyond $\ell$ is of order $\sum_{k>\ell}(\gamma\Delta)^k$, which
\emph{diverges}
whenever $\gamma\Delta>1$---as happens in much of the sub-threshold regime
(e.g.\ $\Delta=3$, $\gamma=0.55$ gives $\gamma\Delta=1.65$ yet
$\kappa=0.91<1$). The saturation rate $\kappa^{\ell/3}$ is invisible to
first-moment arguments: it comes from the exact cancellation of the two
classes' mean messages (\cref{lem:messages}) together with the
second-moment structure of the broadcast process (\cref{lem:ks}), which
together show that the \emph{net} evidence of generation $k$ has second
moment $O(\kappa^k)$. We expect the exponent $\ell/3$ (an artifact of
Chebyshev plus anti-concentration) to be improvable to the natural
$\ell/2$; the experiments in \cref{sec:experiments} indeed show decay
at rate $\kappa^{\ell/2}$, and \cref{thm:fliplower} below proves
the matching $\kappa^{\ell/2}$ \emph{lower} bound at the level of
decision flips. \cref{ass:density} can also be dispensed
with: \cref{app:complements} records the L\'evy-concentration
form of \eqref{eq:below-main} that the proof actually yields---the right
statement for atomic (e.g.\ discrete) feature laws.

\begin{theorem}[decision churn: a matching flip-rate lower bound]
\label{thm:fliplower}
Let $\kappa<1$ and let
\crefrange{ass:sym}{ass:informative} hold
(\cref{ass:density} is not needed). Assume in addition the
regularity: (i) $\sigma_t^2:=\Var_+\big(t(X)\big)>0$, and (ii) the density
of $\log\psi(X)$ under $\mathbb{P}_+$ is bounded below by $\beta>0$ on
$[-K_0,K_0]$, where $K_0:=1+8\sqrt{C_1\kappa\Delta/(\Delta-1)}$. Then
there are $c>0$ and $\ell_0$, explicit in the model constants and the
absolute Berry--Esseen constant, such that
\[
\Prob\big(h_\ell\neq h_{\ell+1}\big)\;\ge\;c\,\kappa^{\ell/2}
\qquad\text{for all }\ell\ge\ell_0.
\]
Both regularity conditions hold, with explicit constants, for the
Gaussian mixture of \cref{ex:gaussian}.
\end{theorem}

\cref{thm:fliplower} certifies that the geometric saturation
happens at exponent exactly $\ell/2$ on the scale of \emph{decisions}:
every layer beyond $\ell_0$ still changes the classifier's output with
probability at least $c\,\kappa^{\ell/2}$, matching the empirically sharp
rate of \cref{fig:flips} and pinning the flip probability between
$c\,\kappa^{\ell/2}$ and twice the right-hand side of
\eqref{eq:below-main}. Three honest caveats. First, in the spirit of
\cref{tab:certify}: the constants $c$ and $\ell_0$ are far from
optimized---$c$ inherits the looseness of $C_1$ through the window
$K_0$, and $\ell_0$ that of the Berry--Esseen step---so the theorem's
content is the exponent, not the constant. Second, the gap between the
exponents ($\ell/2$ from below, $\ell/3$ from above) remains, on the
upper side. Third, a flip is a statement about decision churn, not
about error: layer $\ell+1$ changes the answer at rate $\kappa^{\ell/2}$,
but the two directions of change nearly cancel in $\err(\ell)$, and a
lower bound on $|\err(\ell)-\err(\infty)|$ itself remains open
(\cref{sec:discussion}). The proof---a conditional
Berry--Esseen argument on the newest generation, combined with the
root's density near the decision boundary---is in
\cref{app:complements}.

\subsection{Above the threshold: geometric amplification}

\begin{theorem}[depth amplification, $\kappa>1$]\label{thm:above}
Let $\kappa>1$ and set $c_0=\vartheta^2/(8c_\gamma^2)$. Then for every
$\ell\ge1$ and every $s\in(0,1)$,
\begin{equation}\label{eq:above-main}
\err(\ell) \;\le\; \frac{16}{\kappa-1}
\;+\; 2\,\kappa^{-s\ell}
\;+\; \exp\!\Big(\!-c_0\,(\kappa-1)\,\frac{\kappa^{(1-s)\ell-1}}{\ell^{2}}\Big)
\;+\; \exp\!\Big(\!-\frac{\vartheta}{2}\,\kappa^{\ell}\Big).
\end{equation}
\end{theorem}

Above the threshold each layer multiplies the conditional signal-to-noise
ratio of $T_\ell$ by $\kappa$: the signed excess of generation $k$
concentrates on $W\,(\gamma\Delta)^k$ for the Kesten--Stigum martingale limit
$W$, so the conditional mean of $T_\ell$ grows like $\kappa^\ell$ while its
conditional standard deviation grows like $\kappa^{\ell/2}$. The error is
therefore driven down to a \emph{floor}---at most $16/(\kappa-1)$ in our
bound---reflecting the event
that the branching process fluctuates badly (e.g.\ dies out or produces an
atypically unbalanced early generation)---an event no amount of depth
repairs. Two honest readings of \eqref{eq:above-main} are in order. First,
the geometric term $2\kappa^{-s\ell}$ is not a rate produced by the
signal-to-noise mechanism: it is the Markov-inequality price of truncating
the generation sizes at $\kappa^{s\ell}k^2\Delta^k$, and $s\in(0,1)$ is
free (the genuinely signal-driven terms are the two super-geometric
exponentials). The approach to the floor in our bound is therefore faster
than every geometric rate $\kappa^{-s\ell}$, $s<1$; what the \emph{true}
rate of $\err(\ell)$'s approach to its large-$\ell$ behaviour is---indeed,
whether $\lim_{\ell\to\infty}\err(\ell)$ exists at all for
$\kappa>1$---we do not know: the Cauchy argument of
\cref{thm:below} is below-threshold only, and the empirical curves
of \cref{sec:experiments} drift mildly \emph{upward} beyond a
finite optimal depth, a drift that \eqref{eq:above-main} caps but no
result here eliminates. Second, we have made no attempt to optimize the
floor's constant: as stated, the bound is non-vacuous only for
$\kappa>17$---which, since $\kappa<\Delta$, requires average degree above
$17$---and informative for larger $\kappa$; near the threshold its
degradation mirrors, in exaggerated form, the genuine weakening of
reconstruction on trees near criticality. \cref{sec:experiments}
evaluates both theorems at the simulated parameters.

\subsection{A universal floor and depth prescriptions}

\begin{proposition}[depth-independent floor]\label{prop:floor}
Let $\varepsilon_{\mathrm{feat}}=\tfrac12\int\min(\rho_+,\rho_-)\dd\mu$
denote the Bayes error of the balanced feature mixture. Then for every
$\ell\ge0$ and every classifier $h$ that is a measurable function of the
depth-$\ell$ feature-decorated neighbourhood of the root,
\[
\Prob\big(h\neq y_o\big) \;\ge\; e^{-\Delta}\,\varepsilon_{\mathrm{feat}}.
\]
Moreover, conditionally on the root being isolated, the error of $h_\ell$ is
exactly $\varepsilon_{\mathrm{feat}}$. For the Gaussian mixture of
\cref{ex:gaussian}, $\varepsilon_{\mathrm{feat}}=\Phi(-\zeta)$.
\end{proposition}

\begin{corollary}[how deep is deep enough]\label{cor:depth}
Fix $\epsilon\in(0,1)$.
\begin{enumerate}[leftmargin=2em,itemsep=0.15em]
\item If $\kappa<1$, then every depth
$\ell\ge \dfrac{3\log\!\big(3B^{2/3}C_1^{1/3}/\epsilon\big)}{\log(1/\kappa)}$
satisfies $\err(\ell)\le \err(\infty)+\epsilon$: all further layers combined
are worth less than $\epsilon$.
\item If $\kappa>1$, then $\err(\ell)\le \dfrac{16}{\kappa-1}+\epsilon$
for every $\ell\ge\ell_1$, where $\ell_1=\ell_1(\epsilon)$ is the smallest
depth $\ell\ge4/\log\kappa$ satisfying the following three explicit
conditions (instantiating \eqref{eq:above-main} at $s=\tfrac12$); each
left-hand side is nondecreasing in $\ell$ from $4/\log\kappa$ on, so the
conditions persist for all larger $\ell$:
\[
\kappa^{\ell/2}\ge \frac{6}{\epsilon},\qquad
c_0(\kappa-1)\,\kappa^{\ell/2-1}\ \ge\ \ell^{2}\log\frac{3}{\epsilon},
\qquad
\frac{\vartheta}{2}\,\kappa^{\ell}\ \ge\ \log\frac{3}{\epsilon};
\]
the first is the binding one as $\epsilon\to0$, so
$\ell_1=2\log_\kappa(6/\epsilon)+O_{\mathrm{model}}(\log\log(1/\epsilon))$.
\end{enumerate}
In both regimes, depth $O(\log(1/\epsilon))$ \emph{suffices}, independent
of any graph size. We do not prove matching lower bounds (beyond
\cref{prop:first}); see \cref{sec:discussion}.
\end{corollary}

\begin{proposition}[the first layer provably helps]\label{prop:first}
Let $\tau_{\mathrm{feat}}:=1-2\varepsilon_{\mathrm{feat}}$ denote the
total-variation distance between $\mathbb{P}_+$ and $\mathbb{P}_-$, and set
\[
\tau_1 \;:=\; \big(1-e^{-\Delta}\big)\,\gamma\,\tau_{\mathrm{feat}},
\]
and assume ties are null: $\Prob(T_0=0)=\Prob(T_1=0)=0$ (automatic under
\cref{ass:density}). Then $\err(1)\le\err(0)$, and for every
$\delta\in(0,\tau_1/2)$,
\begin{equation}\label{eq:first-layer}
\err(0)-\err(1)\;\ge\;
\Big(\frac{\tau_1}{2}-\delta\Big)\;
\Prob\Big(\big|\log\psi(X_o)\big|\le
\log\tfrac{1+\delta}{1-\delta}\Big),
\end{equation}
the probability being over the balanced feature mixture. In particular the
first layer strictly improves on the feature-only rule whenever
$\log\psi(X)$ puts mass in a neighbourhood of $0$---as it does for the
Gaussian mixture of \cref{ex:gaussian}, where every factor
of \eqref{eq:first-layer} is explicit.
\end{proposition}

The proposition rests on the fact, recorded in \cref{rem:bp}, that
$h_1$ \emph{is} the Bayes rule for the depth-$1$ observation, so
$\err(1)$ is a Bayes error and can only fall as the observation grows; the
quantitative content of \eqref{eq:first-layer} is that it falls by a
definite amount whenever the first generation carries total-variation
information ($\tau_1>0$) and the root's own evidence leaves the decision
genuinely uncertain with positive probability. Both discounts are real:
$\tau_1$ deflates $\tau_{\mathrm{feat}}$ by the edge signal $\gamma$ and
by the chance $e^{-\Delta}$ of an isolated root, and the second factor
vanishes as $\zeta\to\infty$, when the root's feature already decides and
a layer genuinely has nothing to add. No analogous lower bound is proven
for deeper layers; that gap is what keeps
\cref{cor:depth} one-sided. The proof---a conditional-Bayes
decomposition with a total-variation tensorization over the Poisson
offspring---is in \cref{app:first}, together with a sharper
constant and an account of where the bound's slack lies.

\begin{remark}[transfer to finite graphs]\label{rem:finite}
On $\CSBM(n,\mathbb{P}_\pm,a/n,b/n)$, for any fixed radius $\ell$ the
$\ell$-neighbourhood of a uniformly random vertex converges in distribution
to the limit object of \cref{sec:csbm}, and $h_\ell$ evaluated at a
random vertex has misclassification probability $\err(\ell)+o_n(1)$; more
quantitatively, neighbourhoods of radius $\ell\le c\log n$ (for
$c\log\Delta<1/4$) are trees with high probability, and quantitative
$O(1/\log^2 n)$ error-transfer bounds at such radii are given in
\citet[Thm.~4, Prop.~3.1]{baranwal2023optimality}. One uniformity caveat
is worth stating: the transfer holds at each fixed $\ell$, with a constant
depending on $\ell$, so statements about $\err(\infty)$ have finite-$n$
meaning only through the iterated limit $n\to\infty$ followed by
$\ell\to\infty$. This costs nothing here, because
\cref{cor:depth} prescribes depths $O(\log(1/\epsilon))$ that do
not grow with $n$: at any such fixed depth the analysis applies directly
to large finite graphs, the useful depth being a constant far below the
$O(\log n)$ locality horizon. Our
simulations on graphs with $n=2\times10^4$ (\cref{fig:graph})
confirm agreement within $\pm0.002$ at all depths considered.
\end{remark}

\subsection{Pairwise message passing versus belief propagation}

\begin{remark}[the pairwise rule is linearized BP]\label{rem:bp}
On the limit tree the exact posterior of the root label given the
depth-$\ell$ observation is computed by belief propagation: the log
likelihood-ratio at $v$ from its subtree is
\[
L_v \;=\; \log\psi(X_v) \;+\; \sum_{w\,:\,\text{child of }v}
F_\gamma(L_w),
\qquad F_\gamma(z) \;=\; 2\artanh\!\big(\gamma\tanh(z/2)\big),
\]
and the exact rule is $\sgn(L_o)$. The pairwise statistic
\eqref{eq:classifier} replaces the composition of contractions along each
root-to-vertex path by the \emph{single} contraction
$2\artanh(\gamma^k t(X_v))$, applied to each vertex separately, and then
sums. Because $F_\gamma\big(2\artanh(\gamma^{k-1}t)\big) =
2\artanh(\gamma^{k}t)$, the two rules agree exactly whenever each
root-to-leaf path carries at most one informative vertex other than the
root (whose evidence enters additively, outside every contraction, in both
rules)---in particular at depth $\ell=1$, and on ``single-source'' paths of
any depth. They differ as
soon as evidence stacks along a path, because $F_\gamma$ is applied to a
\emph{sum} of log-likelihood ratios by BP but distributed across summands by
the pairwise rule. A depth-$2$ example with $\gamma=\tfrac12$: a path
$o\!-\!w\!-\!v$ with $\psi(X_w)=3$, $\psi(X_v)=2$ has exact log posterior
ratio $\log\frac{17}{9}\approx0.636$, while the pairwise statistic is
$2\artanh\big(\gamma\,t(X_w)\big)+2\artanh\big(\gamma^2t(X_v)\big)
=2\artanh\big(\tfrac14\big)+2\artanh\big(\tfrac1{12}\big)\approx0.678$.
Under feature laws rich enough to realize such configurations (the
Gaussian mixture among them), the two rules disagree with positive
probability, so for $\ell\ge2$ the pairwise rule is not the exact Bayes
rule for the depth-$\ell$ observation; it is Bayes with respect to the
\emph{pairwise (distance-marginal) likelihood factorization}
\citep[Eq.~(4)]{baranwal2023optimality} in which each
vertex's label is coupled to the root only through the $k$-step transition
matrix. This is the tree-side face of a familiar object: linearizing BP
around its uninformative fixed point on a graph produces the
non-backtracking operator \citep{krzakala2013spectral,
bordenave2018nonbacktracking}, which propagates evidence along
non-backtracking walks with factor $\gamma$ per edge; on a tree every
root-to-vertex path is non-backtracking, and the pairwise rule assigns
each vertex exactly that path product, $\gamma^{k}$, inside its own
contraction. We analyze the pairwise rule because it is the one implemented
by the message-passing architecture of \citet{baranwal2023optimality} and
because its additive structure is what a GNN layer computes; we conjecture
that exact BP obeys the same dichotomy---with, below the threshold, a
strictly faster saturation exponent for which the pairwise rate is only
an upper envelope (\cref{sec:experiments,sec:discussion})---consistent with
robust reconstruction \citep{janson2004robust}.
\end{remark}

\begin{observation}[finite optimal depth: non-monotonicity of the pairwise
rule]\label{obs:nonmono}
The error curve of exact BP is non-increasing in $\ell$ (its depth-$\ell$
decision is Bayes for a filtration that grows with $\ell$); the BP
baseline in \cref{sec:experiments}, run on the same sampled trees,
confirms this within statistical error. The pairwise
rule enjoys no such guarantee, and indeed its error curve is mildly
non-monotone: in the simulations of \cref{sec:experiments} (e.g.\
$\kappa=0.65$, $\zeta=0.6$) the error falls from $0.274$ at $\ell=0$ to a
minimum $0.222$ at $\ell=2$ and then drifts up to $0.237$ by $\ell=6$,
while BP continues downward. The
mechanism is the one exposed in \cref{rem:bp}: messages from vertices
sharing ancestry are positively correlated, but the pairwise rule adds them
as if independent, so beyond the depth at which fresh signal is exhausted,
the statistic keeps accumulating correlated noise at full weight.
\cref{thm:below} bounds the total damage of overshooting any depth
$\ell$ by $3B^{2/3}C_1^{1/3}\kappa^{(\ell+1)/3}$; the practical reading is
that the depth-$\ell^*$ sweet spot is forgiving on the deep side below
threshold, and that architectures implementing linearized message passing
should treat depth as a tuned hyperparameter rather than ``more is better.''
The non-monotonicity itself is not merely an empirical claim:
\cref{app:certified} certifies an exact instance.
\end{observation}

\section{Proofs}\label{sec:proofs}

Throughout, $\mathcal{G}=\sigma(T,\{y_v\}_{v\in T})$ denotes the
$\sigma$-algebra generated by the tree and all labels, and
$D_k=\sum_{v\in N_k}\sigma_v=\alpha_k-\beta_k$ is the signed excess of
generation $k$ (with $\alpha_k,\beta_k$ the numbers of vertices of
generation $k$ whose label agrees, resp.\ disagrees, with the root). We
freely use that, conditionally on $\mathcal{G}$, the features
$\{X_v\}_{v\in T}$ are independent with $X_v\sim\mathbb{P}_{y_v}$.

\subsection{Preliminary lemmas}

\begin{lemma}[symmetry]\label{lem:symmetry}
Under \cref{ass:sym}, for every $\ell$,
$\Prob(y_oT_\ell\le0\,|\,y_o=+1)=\Prob(y_oT_\ell\le0\,|\,y_o=-1)$; hence
$\err(\ell)=\Prob(T_\ell\le0\,|\,y_o=+1)$.
\end{lemma}

\begin{proof}
Consider the map that flips every label ($y_v\mapsto-y_v$) and replaces
every feature by $X_v\mapsto\tau(X_v)$. It preserves the law of the limit
object: the tree shape is untouched, the broadcast process depends only on
the relative labels $\{\sigma_v\}$, which are invariant, the root's label
stays uniform, and $X_v\sim\mathbb{P}_{y_v}$ maps to
$\tau(X_v)\sim\mathbb{P}_{-y_v}$ since $\tau$ pushes $\mathbb{P}_{+}$ to
$\mathbb{P}_{-}$ and vice versa (\cref{ass:sym}). Since
$t(\tau(x))=-t(x)$ by \eqref{eq:t-def}, every message and hence $T_\ell$ is
negated, while $y_o$ is negated as well, so $y_oT_\ell$ is invariant in law
conditionally on either value of $y_o$.
\end{proof}

In the remaining proofs we condition on $y_o=+1$ without further comment, so
that $\sigma_v=y_v$ and $\err(\ell)=\Prob(T_\ell\le0)$.

\begin{lemma}[message moments]\label{lem:messages}
Let $m_k:=\E_+[M_k(X)]$ for $k\ge0$. Under
\crefrange{ass:sym}{ass:informative}, for every $k\ge1$:
\begin{enumerate}[leftmargin=2em,itemsep=0.1em]
\item[(i)] \textup{(antisymmetry)}\; $\E_-[M_k(X)]=-m_k$;
\item[(ii)] \textup{(signal lower bound)}\; $m_k\ \ge\ 2\gamma^k\vartheta$;
\item[(iii)] \textup{(second moment)}\;
$\E_\pm[M_k(X)^2]\ \le\ c_\gamma^2\gamma^{2k}\vartheta$, and consequently
$m_k\le c_\gamma\gamma^k\sqrt{\vartheta}$;
\item[(iv)] \textup{(clipping)}\; $|M_k(x)|\le c_k=2\artanh(\gamma^k)\le
c_\gamma\gamma^k$ for all $x$.
\end{enumerate}
\end{lemma}

\begin{proof}
Since $\tau$ preserves $\mu$ and $\rho_-=\rho_+\circ\tau$, for any bounded
measurable $g$,
\[
\E_-[g(t(X))]=\int g(t(x))\rho_+(\tau(x))\dd\mu(x)
=\int g(t(\tau(x)))\rho_+(x)\dd\mu(x)=\E_+[g(-t(X))],
\]
using $t\circ\tau=-t$. Taking $g=2\artanh(\gamma^k\cdot)$, which is odd,
gives (i). The same computation shows
$\E_-[t^2]=\E_+[t^2]=\vartheta$.

For (ii), write, using (i) in the form
$m_k=\tfrac12\int 2\artanh(\gamma^k t(x))\,(\rho_+(x)-\rho_-(x))\dd\mu(x)$
and the pointwise identity $\rho_+-\rho_-=t\cdot(\rho_++\rho_-)$ from
\eqref{eq:t-def},
\[
m_k=\tfrac12\int 2\artanh\big(\gamma^k t(x)\big)\,t(x)\,
\big(\rho_+(x)+\rho_-(x)\big)\dd\mu(x)
\;\ge\;\tfrac12\int 2\gamma^k t(x)^2(\rho_++\rho_-)\dd\mu
=2\gamma^k\vartheta,
\]
where we used $\artanh(z)\,z\ge z^2$ for $|z|<1$ (as $|\artanh z|\ge|z|$
with equal signs).

For (iii), $|2\artanh(\gamma^kt)|\le 2\gamma^k|t|/(1-\gamma^{2k}t^2)\le
c_\gamma\gamma^k|t|$ for $k\ge1$, using $\artanh z=z+\tfrac{z^3}{3}+\cdots
\le z/(1-z^2)$ for $z\in[0,1)$ and $\gamma^{2k}t^2\le\gamma^2$; square and
take expectations, then apply Cauchy--Schwarz for the bound on $m_k$.
Part (iv) is monotonicity of $\artanh$ together with the same series bound
evaluated at $z=\gamma^k$.
\end{proof}

\begin{lemma}[broadcast and branching structure]\label{lem:ks}
Let $W_k := D_k/(\gamma\Delta)^k$, with $W_0=1$. Then:
\begin{enumerate}[leftmargin=2em,itemsep=0.1em]
\item[(i)] conditionally on the tree shape, the relative labels
$\{\sigma_v\}$ form a Markov broadcast: each edge independently preserves
the label with probability $\tfrac{1+\gamma}{2}$; consequently
$\E[\sigma_v\sigma_w\,|\,T]=\gamma^{d(v,w)}$ for all $v,w\in T$;
\item[(ii)] $\E[S_k]=\Delta^k$;
\item[(iii)] $(W_k)_{k\ge0}$ is a martingale with respect to
$\mathcal{F}_k=\sigma\big(\text{generations }0,\dots,k\big)$, with
$\E[W_k]=1$ and
\[
\E\big[(W_{k+1}-W_k)^2\big]=\kappa^{-(k+1)},\qquad
\E[W_j W_k]=\E[W_{j\wedge k}^2]=1+\sum_{i=1}^{j\wedge k}\kappa^{-i};
\]
\item[(iv)] if $\kappa<1$: $\E[W_j^2]\le \dfrac{\kappa^{-j}}{1-\kappa}$, so
$\E[D_jD_k]\le\dfrac{(\gamma\Delta)^{j+k}\,\kappa^{-(j\wedge k)}}{1-\kappa}$;
\item[(v)] if $\kappa>1$: $\Var(W_j)\le\dfrac{1}{\kappa-1}$ and
$\E[(W_\infty-W_j)^2]=\sum_{i>j}\kappa^{-i}$, where $W_\infty$ is the
almost-sure and $L^2$ limit of $W_j$.
\end{enumerate}
\end{lemma}

\begin{proof}
(i) In the two-type description, a child's relative label equals its
parent's with probability $\frac{a/2}{\Delta}=\frac{1+\gamma}{2}$
independently across children, and the total offspring count
$\Poi(a/2)+\Poi(b/2)=\Poi(\Delta)$ is independent of the types by Poisson
thinning. Multiplying independent edge factors with mean $\gamma$ along the
path from $v$ to $w$ gives $\E[\sigma_v\sigma_w|T]=\gamma^{d(v,w)}$.
(ii) is the standard branching-process first moment.
(iii) Conditionally on $\mathcal{F}_k$, $\alpha_{k+1}$ and $\beta_{k+1}$ are
independent Poissons with means $\tfrac{a\alpha_k+b\beta_k}{2}$ and
$\tfrac{a\beta_k+b\alpha_k}{2}$, so
$\E[D_{k+1}|\mathcal{F}_k]=\tfrac{a-b}{2}D_k=\gamma\Delta\,D_k$ and
$\Var(D_{k+1}|\mathcal{F}_k)=\Delta S_k$. The martingale property follows;
moreover
$\E[(W_{k+1}-W_k)^2]=\E[\Delta S_k]/(\gamma\Delta)^{2(k+1)}
=\Delta^{k+1}/(\gamma\Delta)^{2k+2}=\kappa^{-(k+1)}$, and orthogonality of
martingale increments gives the formula for $\E[W_j^2]$; for $j\le k$,
$\E[W_jW_k]=\E[W_j\E[W_k|\mathcal{F}_j]]=\E[W_j^2]$.
(iv) $1+\sum_{i=1}^{j}\kappa^{-i}
=\frac{\kappa^{-j}-\kappa}{1-\kappa}\le\frac{\kappa^{-j}}{1-\kappa}$,
and $\E[D_jD_k]=(\gamma\Delta)^{j+k}\E[W_{j\wedge k}^2]$.
(v) For $\kappa>1$ the increment variances are summable:
$\Var(W_j)=\sum_{i=1}^{j}\kappa^{-i}\le\frac{1}{\kappa-1}$, and $L^2$
martingale convergence applies; this is the Kesten--Stigum $L^2$ regime
\citep{kesten1966limit}.
\end{proof}

\subsection{Proof of Theorem~\ref{thm:below}}

The heart of the proof is a second-moment bound on the cumulative
contribution of all generations beyond $\ell$.

\begin{lemma}[tail contribution]\label{lem:tail}
Let $\kappa<1$ and, for $\ell<\ell'$, set
$R:=T_{\ell'}-T_\ell$. Then
\[
R=\sum_{k=\ell+1}^{\ell'}\sum_{v\in N_k}M_k(X_v)
\qquad\text{satisfies}\qquad
\E[R^2]\;\le\; C_1\,\kappa^{\ell+1},
\qquad C_1=\frac{3c_\gamma^2\vartheta}{(1-\kappa)^3}.
\]
\end{lemma}

\begin{proof}
Condition on $\mathcal{G}$ and decompose
$\E[R^2]=\E[\Var(R|\mathcal{G})]+\E[(\E[R|\mathcal{G}])^2]$.

\emph{Variance part.} Given $\mathcal{G}$ the summands are independent, so
by \cref{lem:messages}(iii),
\[
\Var(R\,|\,\mathcal{G})=\sum_{k>\ell}\sum_{v\in N_k}\Var\big(M_k(X_v)\,|\,
y_v\big)\;\le\;c_\gamma^2\vartheta\sum_{k>\ell}S_k\gamma^{2k},
\]
whence, by \cref{lem:ks}(ii),
$\E[\Var(R|\mathcal{G})]\le
c_\gamma^2\vartheta\sum_{k>\ell}(\gamma^2\Delta)^k
=c_\gamma^2\vartheta\,\frac{\kappa^{\ell+1}}{1-\kappa}$.

\emph{Mean part.} By \cref{lem:messages}(i),
$\E[M_k(X_v)\,|\,\mathcal{G}]=y_v\,m_k=\sigma_v m_k$ (recall $y_o=+1$), so
$\E[R|\mathcal{G}]=\sum_{k>\ell}m_kD_k$. Using
$m_jm_k\le c_\gamma^2\vartheta\,\gamma^{j+k}$
(\cref{lem:messages}(iii)) and \cref{lem:ks}(iv),
\[
\E\big[(\E[R|\mathcal{G}])^2\big]
=\sum_{j,k>\ell}m_jm_k\,\E[D_jD_k]
\;\le\;\frac{c_\gamma^2\vartheta}{1-\kappa}
\sum_{j,k>\ell}\gamma^{j+k}(\gamma\Delta)^{j+k}\kappa^{-(j\wedge k)}
=\frac{c_\gamma^2\vartheta}{1-\kappa}\sum_{j,k>\ell}\kappa^{\max(j,k)},
\]
where the last equality uses
$\gamma^{j+k}(\gamma\Delta)^{j+k}=(\gamma^2\Delta)^{j+k}=\kappa^{j+k}$ and
$\kappa^{j+k-(j\wedge k)}=\kappa^{\max(j,k)}$. The number of pairs
$(j,k)$ with $j,k>\ell$ and $\max(j,k)=m$ is $2(m-\ell)-1\le2(m-\ell)$, so
\[
\sum_{j,k>\ell}\kappa^{\max(j,k)}\le 2\sum_{m>\ell}(m-\ell)\kappa^{m}
=2\kappa^{\ell}\sum_{r\ge1}r\kappa^{r}
=\frac{2\kappa^{\ell+1}}{(1-\kappa)^{2}}.
\]
Combining, $\E[(\E[R|\mathcal{G}])^2]\le
\frac{2c_\gamma^2\vartheta\,\kappa^{\ell+1}}{(1-\kappa)^{3}}$, and adding
the variance part (using $\tfrac{1}{1-\kappa}\le\tfrac{1}{(1-\kappa)^3}$)
gives the claim.
\end{proof}

\begin{lemma}[anti-concentration]\label{lem:anticonc}
Under \cref{ass:density}, for every $\ell'\ge0$ and $s>0$, we have that
$\Prob\big(|T_{\ell'}|\le s\big)\le 2Bs$.
\end{lemma}

\begin{proof}
Write $T_{\ell'}=\log\psi(X_o)+U$ where $U$ collects the contributions of
all non-root vertices. Conditionally on $y_o$, the feature $X_o$ is
independent of $U$ (which is a function of the tree, the other labels, and
the other features). Under $\mathbb{P}_+$, $\log\psi(X)$ has density
bounded by $B$; under $\mathbb{P}_-$ the law of $\log\psi(X)$ is the
pushforward of the former under $z\mapsto-z$ (by \cref{ass:sym},
since $\log\psi\circ\tau=-\log\psi$), so its density obeys the same bound.
Hence, conditioning on $(y_o,U)$,
$\Prob\big(\log\psi(X_o)\in[-s-U,\,s-U]\,\big|\,y_o,U\big)\le 2Bs$, and the
claim follows by averaging.
\end{proof}

\begin{proof}[Proof of \cref{thm:below}]
Fix $0\le\ell<\ell'$ and abbreviate $R=T_{\ell'}-T_\ell$. For any $s>0$,
\begin{align*}
\err(\ell)-\err(\ell')
&=\Prob(T_\ell\le0)-\Prob(T_{\ell'}\le0)
\;\le\;\Prob\big(T_\ell\le0,\;T_{\ell'}>0\big)\\
&\le\;\Prob\big(0<T_{\ell'}\le s\big)
+\Prob\big(T_\ell\le 0,\;T_{\ell'}>s\big)
\;\le\;2Bs+\Prob\big(|R|>s\big),
\end{align*}
because on the second event $R=T_{\ell'}-T_\ell> s$. Symmetrically,
$\err(\ell')-\err(\ell)\le\Prob(-s< T_{\ell'}\le0)+\Prob(T_{\ell'}\le-s,\,
T_\ell>0)\le 2Bs+\Prob(|R|>s)$. By Chebyshev's inequality and
\cref{lem:tail}, $\Prob(|R|>s)\le C_1\kappa^{\ell+1}/s^2$, so
\[
\big|\err(\ell)-\err(\ell')\big|\;\le\;
2Bs+\frac{C_1\kappa^{\ell+1}}{s^{2}}\qquad\text{for every }s>0.
\]
Choosing $s=\big(C_1\kappa^{\ell+1}/B\big)^{1/3}$ balances the two terms
and yields \eqref{eq:below-main}. The bound is uniform in $\ell'$ and
vanishes as $\ell\to\infty$, so $(\err(\ell))_{\ell\ge0}$ is a Cauchy
sequence; letting $\ell'\to\infty$ in \eqref{eq:below-main} gives the
statement about $\err(\infty)$.

For the identification of the limit, \cref{lem:tail} says
$\E[(T_{\ell'}-T_\ell)^2]\le C_1\kappa^{\ell+1}$, so $(T_\ell)$ is Cauchy
in $L^2$ and converges to some $T_\infty$ with
$\E[(T_\infty-T_\ell)^2]\le C_1\kappa^{\ell+1}$. The anti-concentration
\cref{lem:anticonc} holds for $T_\infty$ as well (its proof only
uses that $T_\infty-\log\psi(X_o)$ is independent of $X_o$ given $y_o$,
which is preserved in the limit), whence, for every $s>0$,
\[
\big|\Prob(T_\ell\le0)-\Prob(T_\infty\le0)\big|
\;\le\;2Bs+\Prob\big(|T_\infty-T_\ell|>s\big)
\;\le\;2Bs+\frac{C_1\kappa^{\ell+1}}{s^{2}}\;\xrightarrow[\ell\to\infty]{}\;0
\]
(optimizing $s$ as above), so
$\err(\infty)=\lim_\ell\Prob(T_\ell\le0)=\Prob(T_\infty\le0)
=\Prob(y_oT_\infty\le0)$ by \cref{lem:symmetry}.
\end{proof}

\subsection{Proof of Theorem~\ref{thm:above}}

\begin{proof}[Proof of \cref{thm:above}]
Fix $s\in(0,1)$. Condition on $y_o=+1$ (\cref{lem:symmetry}) and split
$T_\ell=\log\psi(X_o)+T'$, where
$T'=\sum_{k=1}^{\ell}\sum_{v\in N_k}M_k(X_v)$ collects the non-root terms;
all of these are bounded (\cref{lem:messages}(iv)), and conditionally
on $\mathcal{G}$ they are independent with conditional mean
\[
\mu' \;:=\; \E[T'\,|\,\mathcal{G}] \;=\; \sum_{k=1}^{\ell}m_kD_k .
\]

\emph{Step 1: a good event for the branching process.} Let
\[
E_W=\Big\{W_1\ge\tfrac34\Big\}\cap
\Big\{\sup_{k\ge2}|W_k-W_1|\le\tfrac14\Big\},
\qquad
E_S=\Big\{S_k\le K k^2\Delta^k\ \ \forall\,1\le k\le\ell\Big\},
\]
with $K=\kappa^{s\ell}$. On $E_W$ we have $W_k\ge\tfrac12$, i.e.\
$D_k\ge\tfrac12(\gamma\Delta)^k$, for all $1\le k\le\ell$. By Chebyshev and
\cref{lem:ks}(v), $\Prob(W_1<\tfrac34)\le16\Var(W_1)=16\kappa^{-1}$;
by Doob's weak-type maximal inequality applied to the nonnegative
submartingale $\big((W_k-W_1)^2\big)_{k\ge1}$ at level $\tfrac1{16}$,
$\Prob\big(\sup_{k\ge2}|W_k-W_1|>\tfrac14\big)\le
16\sum_{i\ge2}\kappa^{-i}=\frac{16}{\kappa(\kappa-1)}$. Hence
\[
\Prob(E_W^{c})\;\le\;\frac{16}{\kappa}+\frac{16}{\kappa(\kappa-1)}
=\frac{16}{\kappa}\cdot\frac{\kappa}{\kappa-1}=\frac{16}{\kappa-1}.
\]
By Markov's inequality and \cref{lem:ks}(ii),
$\Prob(E_S^{c})\le\sum_{k=1}^{\ell}\frac{\Delta^k}{Kk^2\Delta^k}
\le\frac{\pi^2}{6K}\le 2\kappa^{-s\ell}$. (The choice of the truncation
level $K$ is the only place $s$ enters.)

\emph{Step 2: the conditional mean is large on the good event.} On $E_W$,
using \cref{lem:messages}(ii),
\[
\mu'=\sum_{k=1}^{\ell}m_kD_k\;\ge\;\sum_{k=1}^{\ell}
2\gamma^k\vartheta\cdot\tfrac12(\gamma\Delta)^k
=\vartheta\sum_{k=1}^{\ell}\kappa^{k}\;\ge\;\vartheta\,\kappa^{\ell}.
\]
(All terms are positive, so keeping only $k=\ell$ suffices.)

\emph{Step 3: conditional concentration.} On the event
$\{T_\ell\le0\}$ we must have $\log\psi(X_o)\le-\mu'/2$ or
$T'\le\mu'/2$. Conditionally on $\mathcal{G}$, $T'$ is a sum of independent
random variables with ranges of width $2c_k\le2c_\gamma\gamma^k$ each, so by
Hoeffding's inequality, on $E_W\cap E_S$,
\[
\Prob\Big(T'\le\frac{\mu'}{2}\,\Big|\,\mathcal{G}\Big)
\le\exp\!\Big(\!-\frac{(\mu'/2)^2\cdot 2}{4\sum_{k=1}^{\ell}S_kc_k^2}\Big)
\le\exp\!\Big(\!-\frac{\mu'^2}{8c_\gamma^2\sum_{k=1}^{\ell}S_k
\gamma^{2k}}\Big).
\]
On $E_S$,
$\sum_{k=1}^{\ell}S_k\gamma^{2k}\le K\sum_{k=1}^{\ell}k^2\kappa^{k}\le
K\ell^2\frac{\kappa^{\ell+1}}{\kappa-1}$, so with
$\mu'\ge\vartheta\kappa^\ell$ and $K=\kappa^{s\ell}$,
\[
\Prob\Big(T'\le\frac{\mu'}{2}\,\Big|\,\mathcal{G}\Big)
\;\le\;\exp\!\Big(\!-\frac{\vartheta^2}{8c_\gamma^2}\,
(\kappa-1)\,\frac{\kappa^{(1-s)\ell-1}}{\ell^{2}}\Big)
\;=\;\exp\!\Big(\!-c_0(\kappa-1)\frac{\kappa^{(1-s)\ell-1}}{\ell^2}\Big).
\]
For the root term, $X_o$ is independent of $\mathcal{G}$, and for any
$u\ge0$ the change-of-measure bound
$\Prob_+\big(\log\psi(X)\le -u\big)=\E_-\big[\psi(X)\,
\mathbf{1}\{\psi(X)\le e^{-u}\}\big]\le e^{-u}$ gives, on the good event,
$\Prob\big(\log\psi(X_o)\le-\mu'/2\,|\,\mathcal{G}\big)\le
e^{-\vartheta\kappa^{\ell}/2}$.

\emph{Step 4: assemble.} Combining,
\[
\err(\ell)\le\Prob(E_W^c)+\Prob(E_S^c)
+\E\Big[\mathbf{1}_{E_W\cap E_S}\Big(
\Prob\big(T'\le\tfrac{\mu'}{2}\big|\mathcal{G}\big)
+\Prob\big(\log\psi(X_o)\le-\tfrac{\mu'}{2}\big|\mathcal{G}\big)\Big)\Big],
\]
which is bounded by the four terms of \eqref{eq:above-main}.
\end{proof}

\subsection{Proof of Proposition~\ref{prop:floor}}

\begin{proof}
The root is isolated ($S_1=0$) with probability
$\Prob(\Poi(\Delta)=0)=e^{-\Delta}$, an event determined by the tree alone
and hence independent of $(y_o,X_o)$. Conditionally on $\{S_1=0\}$, the
depth-$\ell$ observation reduces to $X_o$ (all $N_k$, $k\ge1$, are empty),
the label prior remains uniform, and $X_o\sim\mathbb{P}_{y_o}$; therefore
any classifier's conditional error is at least the Bayes error of the
balanced mixture, which equals
$\int\min\big(\tfrac12\rho_+,\tfrac12\rho_-\big)\dd\mu
=\varepsilon_{\mathrm{feat}}$. Multiplying by $\Prob(S_1=0)$ gives the desired bound. For $h_\ell$, conditionally on isolation
$T_\ell=\log\psi(X_o)$, whose sign is the Bayes rule for the mixture, so
the conditional error is exactly $\varepsilon_{\mathrm{feat}}$. For the
Gaussian mixture the Bayes error is
$\Phi\big(-\|\boldsymbol{\mu}\|/\sigma\big)=\Phi(-\zeta)$.
\end{proof}

\begin{proof}[Proof of \cref{cor:depth}]
Part 1 is immediate from \cref{thm:below} with
$3B^{2/3}C_1^{1/3}\kappa^{(\ell+1)/3}\le\epsilon$. For part 2, apply
\eqref{eq:above-main} with $s=\tfrac12$: the three stated conditions are
precisely equivalent to the three $\ell$-dependent terms being at most
$\epsilon/3$ each, and each condition holds for all $\ell$ large enough
(the left-hand sides grow geometrically in $\ell$ against at most
polynomial right-hand sides). Monotonicity from $4/\log\kappa$ on
follows since $\ell\mapsto\tfrac{\ell}{2}\log\kappa-2\log\ell$ is
nondecreasing there; hence $\ell_1$ is well defined and the conclusion
holds for every $\ell\ge\ell_1$. The first
condition gives $\ell\ge2\log_\kappa(6/\epsilon)$ and dominates as
$\epsilon\to0$; the second is satisfied once
$\kappa^{\ell/2}\ge\big(\kappa\ell^2/(c_0(\kappa-1))\big)\log(3/\epsilon)$,
which adds only $O_{\mathrm{model}}(\log\log(1/\epsilon))$; likewise the
third.
\end{proof}

\section{Numerical illustrations}\label{sec:experiments}

We simulate the limit object directly: broadcast-labelled $\PGW(\Delta)$
trees with Gaussian features (\cref{ex:gaussian}), for which
$\log\psi(X_v)\,|\,y_v\sim\mathcal{N}(2\zeta^2y_v,4\zeta^2)$ can be sampled
as a scalar. On every sampled tree we evaluate both the pairwise rule
$h_\ell$ and \emph{exact belief propagation} truncated at depth $\ell$
(the recursion $F_\gamma$ of \cref{rem:bp}, computed bottom-up), so
the two rules are compared on identical randomness. Unless stated
otherwise $\Delta=3$ and each point averages
$4\times10^5$ independent trees ($6\times10^5$ for the flip probabilities
in \cref{fig:flips}); two-standard-error bands are drawn but are
typically thinner than the lines. Code to reproduce all figures accompanies
the paper.

\begin{figure}[t]
\centering
\includegraphics[width=\linewidth]{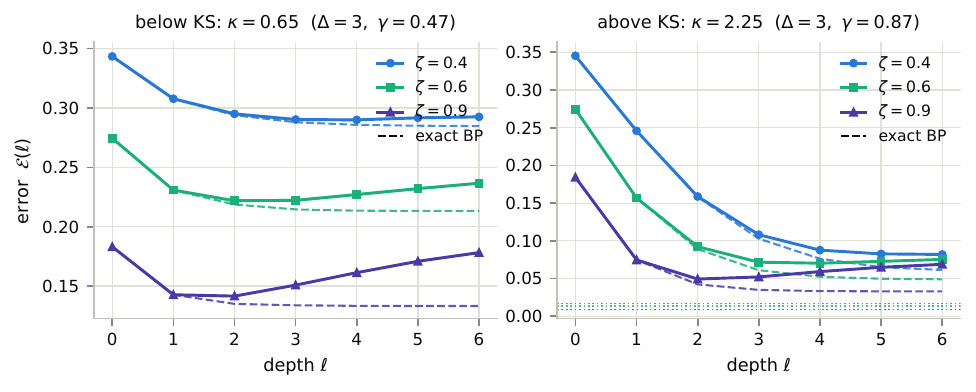}
\caption{Error against depth $\ell$ on the limit tree for the pairwise
rule $h_\ell$ (solid, filled markers) and exact BP truncated at depth
$\ell$ (dashed), computed on the same trees, for
feature strengths $\zeta\in\{0.4,0.6,0.9\}$ ($\Delta=3$; $4\times10^5$
trees per point; shaded bands are $\pm2$ standard errors). \textbf{Left}
(below threshold, $\kappa=0.65$): the pairwise curves saturate within
$2$--$3$ layers and then drift mildly upward
(\cref{obs:nonmono}), while the BP curves are monotone.
\textbf{Right} (above threshold, $\kappa=2.25$): depth is productive and
the error decays geometrically toward a floor; dotted lines mark the
universal floor $e^{-\Delta}\Phi(-\zeta)$ of
\cref{prop:floor}. BP curves are drawn without uncertainty
bands; their standard errors match the pairwise ones.}
\label{fig:curves}
\end{figure}

\paragraph{The dichotomy.} \cref{fig:curves} displays
$\err(\ell)$ for $\ell=0,\dots,6$ in a below-threshold configuration
($\gamma=0.467$, $\kappa=0.65$) and an above-threshold one
($\gamma=0.867$, $\kappa=2.25$). Below threshold, virtually the entire
benefit of the graph is realized by depth $2$: for $\zeta=0.6$ the error
falls $0.274\to0.222$ by $\ell=2$ and moves by less than $0.016$ over all
deeper layers combined. Above threshold, depth keeps paying: the same
feature strength yields $0.274\to0.070$ by $\ell=4$, four times the
below-threshold benefit, consistent with the geometric amplification of
\cref{thm:above}.

\paragraph{The exact-BP baseline and the price of linearization.}
The dashed curves in \cref{fig:curves} run exact BP on the same
trees. Three facts emerge. First, BP's curves are non-increasing in depth
at every configuration, within statistical error---the monotonicity that
\cref{obs:nonmono} attributes to Bayes optimality, verified
empirically. Second, the two rules are indistinguishable through
$\ell=1$ (as they must be; \cref{rem:bp}) and separate from
$\ell=2$ on: the \emph{price of linearization}
$\err(\ell)-\err_{\mathrm{BP}}(\ell)$ at $\zeta=0.6$ grows to
$0.023$ (below threshold) and $0.026$ (above) by $\ell=6$. Its
composition differs by regime: below threshold the gap is mostly the
pairwise rule's upward drift ($+0.015$ from $\ell=2$ to $6$, against
BP's further $-0.005$), while above threshold it is mostly BP's
continued improvement toward a strictly lower floor ($-0.040$, against
the pairwise rule's late $+0.005$ drift). Third, BP's own depth curve
saturates geometrically below threshold and races to its floor
above; its saturation \emph{rate} is measured directly through its
decision flips in \cref{fig:flips} and analyzed below---it is
strictly faster than the pairwise rule's, which refines the conjecture
of \cref{sec:discussion}.

\begin{figure}[t]
\centering
\begin{minipage}[t]{0.48\linewidth}
\centering
\includegraphics[width=\linewidth]{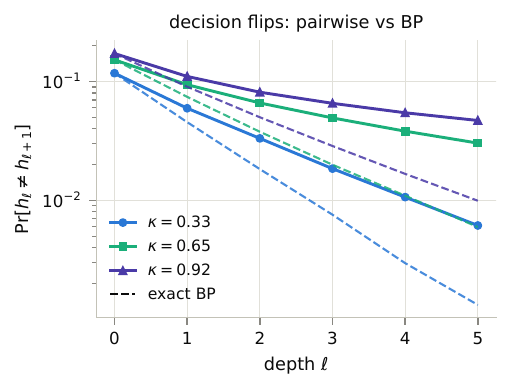}
\end{minipage}\hfill
\begin{minipage}[t]{0.48\linewidth}
\centering
\includegraphics[width=\linewidth]{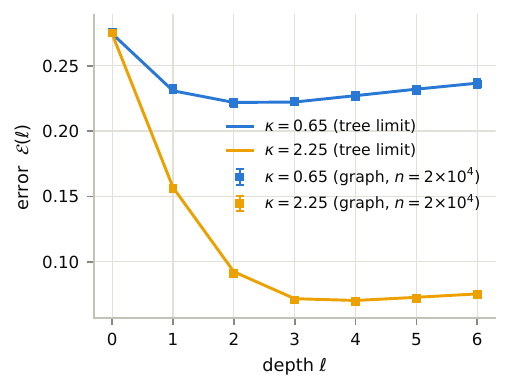}
\end{minipage}
\caption{\textbf{Left:} probability that one more layer changes the
decision, below the threshold
($\Delta=3$, $\zeta=0.6$; $6\times10^5$ trees): pairwise rule (solid)
and exact BP (dashed), on the same trees. Both decay geometrically, BP
strictly faster; fitted rates, and the effective-ratio prediction for
BP, are reported in the text. \textbf{Right:} the tree-limit curves
(lines) against the classifier run on finite sparse CSBM graphs with
$n=2\times10^4$ (markers; mean $\pm$ one standard deviation over $12$
graphs, all vertices classified---note that errors of vertices within one
graph are correlated, so these bars measure graph-to-graph variability,
not estimator error), for $\zeta=0.6$. Finite-graph
means match the limit within $\pm0.002$ at every depth shown.}
\label{fig:flips}\label{fig:graph}
\end{figure}

\paragraph{Geometric saturation at rate $\sqrt{\kappa}$.}
\cref{fig:flips} (left) examines below-threshold saturation directly
through the decision-flip probability $\Prob[h_\ell\ne h_{\ell+1}]$. (The
flip event is the union of the two one-sided events in the proof of
\cref{thm:below}, so it is bounded by \emph{twice} the right-hand
side of \eqref{eq:below-main}.) The decay is cleanly geometric over the
range simulated. Least-squares rates fitted to the last four points give
$0.570$ at $\kappa=1/3$, $0.772$ at $\kappa=0.653$, and $0.833$ at
$\kappa=0.923$, against $\sqrt{\kappa}=0.577$, $0.808$, $0.961$
respectively: the fitted rate matches $\sqrt{\kappa}$ well away from the
threshold (for $\kappa=1/3$ the successive ratios
$0.51,0.56,0.56,0.58,0.58$ climb monotonically toward $0.577$), while at
$\kappa=0.923$ the asymptotic rate has visibly not yet set in by
$\ell=5$---the transient window grows as $\kappa\uparrow1$, consistent
with the near-critical behaviour examined below. The evidence supports
the conjecture that the exponent $\ell/3$ of \cref{thm:below} can
be improved to $\ell/2$ but not beyond, away from the critical window;
\cref{thm:fliplower} proves the matching $\kappa^{\ell/2}$ lower
bound for these flip probabilities, so the exponent seen here is pinned
from below and conjectural only from above.

\paragraph{BP saturates faster: an effective per-layer ratio.}
The dashed curves in \cref{fig:flips} (left) report the same
decision-flip probabilities for exact BP, on the same trees. They too
decay geometrically---but strictly faster than the pairwise rule's,
with fitted rates $0.41$, $0.54$, $0.58$ at
$\kappa=1/3,\,0.653,\,0.923$, far below $\sqrt{\kappa}$. The mechanism
is visible in the recursion: the pairwise rule is the linearization of
$F_\gamma$ at the origin, where its slope $\gamma$ is \emph{maximal};
under exact BP the incoming messages carry $O(1)$ feature evidence, so
the derivatives that propagate a deep perturbation to the root are
evaluated away from zero, where $|F_\gamma'(L)|<\gamma$. The natural
annealed prediction for BP's flip rate is therefore
$\kappa_{\mathrm{BP}}^{\ell/2}$ with
\[
\kappa_{\mathrm{BP}} \;=\; \Delta\,\E\big[F_\gamma'(L)^2\big],
\]
the expectation under the stationary law of the BP message $L$ on the
feature-decorated tree. Computing $\kappa_{\mathrm{BP}}$ by population
dynamics (pool $2\times10^6$, $20$ iterations; code included) gives
$\sqrt{\kappa_{\mathrm{BP}}}=0.424,\,0.559,\,0.621$ at the three
configurations---in close agreement with the measured BP rates, which
climb toward these values from below with the same transient shape as
the pairwise flips approach $\sqrt{\kappa}$. A deeper run at
$\kappa=0.653$ ($2\times10^5$ trees to $\ell=9$) settles the
asymptotics: fitted over
$\ell=5,\dots,8$, the pairwise rate is $0.811$ (against
$\sqrt{\kappa}=0.808$) and the BP rate is $0.551$ (against
$\sqrt{\kappa_{\mathrm{BP}}}=0.559$)---each rule lands on its own
predicted exponent.
Linearization thus carries a rate-level cost alongside the pointwise
one: in our simulations the pairwise rule is less accurate than BP at
every depth beyond $1$, \emph{and} its decisions keep churning at the
envelope rate $\kappa^{\ell/2}$ while BP's settle at the strictly
faster $\kappa_{\mathrm{BP}}^{\ell/2}$. Three qualifications and one
prediction. First, the existence of the stationary message law (hence
of the expectation defining $\kappa_{\mathrm{BP}}$) is an empirical
fact here---the population dynamics converges rapidly and
reproducibly---not a theorem; we do not prove the underlying recursive
distributional equation has a unique attracting fixed point. Second,
the derivation is annealed: it treats the derivative factors along a
path as independent, so $\kappa_{\mathrm{BP}}$ is a prediction whose
status is quantitative agreement, not proof. Third, the object itself
has a lineage: contraction coefficients of the form
$\Delta\,\E[F_\gamma'(L)^2]$, evaluated at a BP fixed point, are
spin-glass-susceptibility--type stability parameters of the kind used
in the reconstruction analysis of \citet{mezard2006reconstruction};
$\kappa_{\mathrm{BP}}$ is that stability parameter at the
\emph{informative}, feature-dressed fixed point, whereas the bare
$\kappa$ is its value at the uninformative one. The prediction:
$\kappa_{\mathrm{BP}}$ and $\kappa$ can sit on opposite sides of $1$.
At the above-threshold configuration of \cref{fig:curves}
($\kappa=2.25$, $\zeta=0.6$) the population dynamics gives
$\kappa_{\mathrm{BP}}=0.356<1$: BP's decision flips should then be
summable---its decision converges---even as depth continues to improve
the error toward the floor. Churn and improvement decouple above
threshold.
\cref{sec:discussion} restates the BP conjecture accordingly.

\paragraph{Finite graphs.} \cref{fig:graph} (right) repeats the
experiment on $12$ finite sparse CSBM graphs with $n=2\times10^4$ vertices,
classifying every vertex with the pairwise rule computed on true BFS
distances. (The sampler Poissonizes the intra- and inter-class pair
counts and draws pairs with replacement, collapsing collisions; the
expected discrepancy from the Bernoulli CSBM is $O(1)$ edges out of
$\approx2\times10^4$ at this size.) The graph means agree with the tree
limit within $\pm0.002$
uniformly over depths $\ell\le6$, illustrating \cref{rem:finite}:
depth conclusions drawn on the limit transfer directly to graphs of
practical size.

\paragraph{Non-monotonicity.} In both regimes the empirical curves attain
their minimum at a finite depth ($\ell^*=2$ below threshold, $\ell^*=4$
above, at $\zeta=0.6$) and then drift upward: e.g.\ $0.222\to0.237$
(below) and $0.070\to0.075$ (above) by $\ell=6$. The effect is small---as
\cref{thm:below} says it must be below threshold---but systematic,
and it grows with feature strength (at $\zeta=0.9$, below threshold:
$0.142$ at $\ell^*=2$ versus $0.178$ at $\ell=6$), consistent with the
correlation double-counting mechanism of \cref{obs:nonmono}:
stronger features make the correlated part of deep messages relatively
more damaging.

\begin{figure}[t]
\centering
\includegraphics[width=\linewidth]{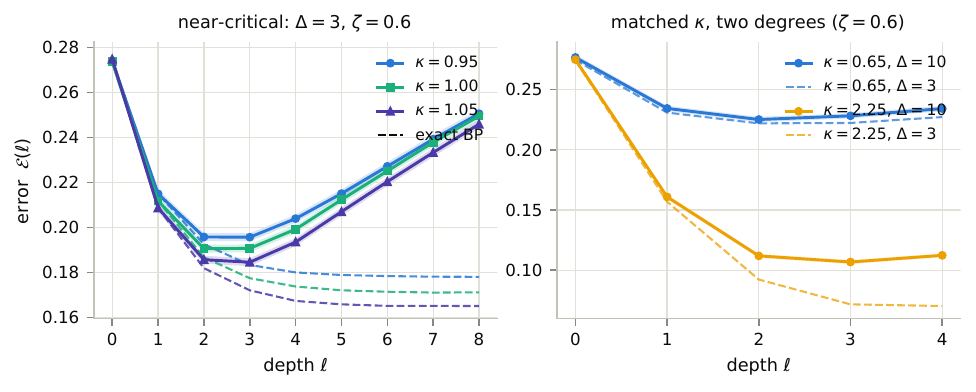}
\caption{\textbf{Left:} the near-critical window: $\err(\ell)$ for
$\kappa\in\{0.95,1.00,1.05\}$ ($\Delta=3$, $\zeta=0.6$, $2\times10^5$
trees per point, depths to $\ell=8$; bands are $\pm2$ standard errors).
The three pairwise curves have the same shape---minima $0.196$, $0.191$,
$0.185$ at $\ell\in\{2,3\}$, then matched upward drift---and nothing at
these depths signals which side of the threshold each sits on: the
dichotomy is an asymptotic statement, and near
$\kappa=1$ a critical window dominates all practical depths. Dashed:
exact BP on the same trees, monotone and flattening---separating the
pairwise rule's correlation double-counting (the drift) from genuine
exhaustion of fresh signal.
\textbf{Right:} a second degree ($1.5\times10^5$ trees per point).
Solid: $\Delta=10$ at the \emph{same}
Kesten--Stigum ratios as \cref{fig:curves}
($\kappa=0.65$ below, $2.25$ above; $\gamma$ adjusted accordingly,
$\zeta=0.6$). Dashed: the corresponding $\Delta=3$ curves. The dichotomy
is driven by $\kappa$, not by the degree: at matched $\kappa$ the
below-threshold curves saturate by $\ell=2$ at both degrees, while the
above-threshold curves fall several times further before the pairwise
drift sets in.}
\label{fig:critical}
\end{figure}

\paragraph{The near-critical window, and a second degree.}
Both theorems degenerate as $\kappa\to1$ ($C_1\sim(1-\kappa)^{-3}$ below;
the floor $\sim(\kappa-1)^{-1}$ above), so the near-critical regime is
where the theory says least. \cref{fig:critical} (left) probes it
directly: at $\kappa\in\{0.95,1.00,1.05\}$ the depth curves are
essentially interchangeable over $\ell\le8$---each falls to a minimum of
$0.185$--$0.196$ at $\ell^*\in\{2,3\}$ and then drifts upward, the
pairwise rule's correlation
double-counting swamping the vanishing fresh signal on both sides of the
threshold; by $\ell=8$ the drift has undone most of the graph's benefit
(errors $\approx0.25$ against $0.274$ at depth $0$). Exact BP on the
same trees (dashed) is monotone and essentially flat beyond
$\ell\approx4$: inside the window it is the pairwise rule's
correlation double-counting, not exhaustion of usable signal, that
shapes the curves. The honest reading:
the dichotomy of
\crefrange{thm:below}{thm:above} is invisible at practical
depths inside a critical window around $\kappa=1$. The width of that
window is in fact what the theorems themselves predict. Below
threshold, \eqref{eq:below-main} has content only once
$3B^{2/3}C_1^{1/3}\kappa^{\ell/3}\lesssim1$; since
$C_1^{1/3}\sim(1-\kappa)^{-1}$ and $\log(1/\kappa)\sim1-\kappa$ as
$\kappa\uparrow1$, the required depth scales as
$\ell\gtrsim 3\big(\log\tfrac{1}{1-\kappa}+O(1)\big)/(1-\kappa)$;
inverted, at depth $\ell$ the saturation statement bites only for
$1-\kappa\gtrsim c\log\ell/\ell$. Above threshold the floor
$16/(\kappa-1)$ exceeds~$1$ until $\kappa-1$ is of constant order.
At $|\kappa-1|=0.05$ and $\ell\le8$, \cref{fig:critical} (left)
sits squarely inside the predicted window on both sides, so the
observed flatness is the shape the theory predicts, not merely a
limitation of it. At $\kappa=1$ exactly
we conjecture polynomial-in-$\ell$ saturation, by analogy with critical
branching processes; the flat minimum in the data is consistent with
this but our depths cannot distinguish polynomial from slow geometric
behaviour. \cref{fig:critical} (right) varies the degree instead:
at $\Delta=10$ with $\gamma$ rescaled to keep $\kappa$ at the values of
\cref{fig:curves}, the same dichotomy appears---evidence that
$\kappa$, and not $(\gamma,\Delta)$ separately, organizes the value of
depth, as the theory asserts.

\paragraph{What the theorems certify at these parameters.}
The rates in \cref{thm:below,thm:above} are the
contribution of this paper; their \emph{constants} are not, and a reader
should know what the bounds certify at the parameters simulated above.
The answer is: little. \cref{tab:certify} evaluates them at
$\zeta=0.6$. Below threshold, the prefactor $3B^{2/3}C_1^{1/3}\approx7.3$
means \eqref{eq:below-main} drops below $\tfrac12$ only at $\ell\approx18$
and certifies $\epsilon=0.01$ saturation only at $\ell\approx45$, against
an empirical saturation depth of $2$--$3$; above threshold,
$\kappa=2.25<17$ makes the floor term alone exceed $1$ (indeed
$\kappa<\Delta=3$, so \emph{no} sparse configuration at this degree is
covered non-trivially by \cref{thm:above}). The experiments are
thus consistent with, but in no sense certified by, the theorems at these
parameters. What the data support is that the \emph{below-threshold
saturation rate} is sharp in the exponent ($\kappa^{\ell/2}$,
\cref{fig:flips}); above threshold no rate is measured---the
geometric term of \eqref{eq:above-main} is a truncation artifact and the
floor's location is unknown---and the constants are far from sharp in
both regimes. The one non-vacuous certification at these
parameters is \cref{prop:first}: evaluating
\eqref{eq:first-layer} at $\delta=\tau_1/4$ gives
$\err(0)-\err(1)\ge2.8\times10^{-3}$ below threshold and
$\ge9.6\times10^{-3}$ above, against measured drops of $0.044$ and
$0.118$.

\begin{table}[t]
\centering
\small
\begin{tabular}{@{}lcc@{}}
\toprule
& below KS ($\kappa=0.65$) & above KS ($\kappa=2.25$)\\
\midrule
\cref{thm:below}: prefactor $3B^{2/3}C_1^{1/3}$ & $7.3$ & ---\\
\quad depth for bound $\le\tfrac12$ / $\le0.01$ & $18$ / $45$ & ---\\
\quad empirical saturation depth & $2$--$3$ & ---\\
\cref{thm:above}: floor $16/(\kappa-1)$ & --- & $12.8$ ($>1$: vacuous)\\
\cref{prop:first}: certified $\err(0)-\err(1)$ &
$\ge0.0028$ & $\ge0.0096$\\
\quad measured $\err(0)-\err(1)$ & $0.044$ & $0.118$\\
\bottomrule
\end{tabular}
\caption{What the bounds certify at the simulated parameters ($\Delta=3$,
$\zeta=0.6$; Gaussian features, for which $B=(2\zeta\sqrt{2\pi})^{-1}
\approx0.33$, $\vartheta\approx0.27$, $C_1\approx1.3\times10^{2}$).}
\label{tab:certify}
\end{table}

\section{Discussion and open problems}\label{sec:discussion}

\paragraph{Design implications.} For sparse graphs, the analysis supports
three design rules. (1)~\emph{Depth should be logarithmic in the target
accuracy, not in the graph size}: $O(\log(1/\epsilon))$ layers suffice in
both regimes, and below the Kesten--Stigum threshold anything beyond a few
layers is provably immaterial for the pairwise rule. (2)~\emph{Attenuation,
not averaging}: the statistically-derived message at distance $k$ is
clipped to a window of width $O(\gamma^k)$; uniform neighbourhood averaging
(as in vanilla deep GCN stacks) ignores this schedule, and its well-known
degradation with depth may be read as the price of that mismatch.
(3)~\emph{Depth is no cure for weak features}: the floor
$e^{-\Delta}\varepsilon_{\mathrm{feat}}$ binds all local classifiers.

\paragraph{Sharpness, and what is missing.} We believe the saturation
exponent $\ell/3$ in
\cref{thm:below} is improvable to $\ell/2$ (matching the
simulations, and matching the flip-rate lower bound of
\cref{thm:fliplower}). The bottleneck is specifically the mean
part of the tail statistic: conditionally on the tree and labels, the
increment $R$ is a sum of bounded terms whose quadratic variation
$\sum_{k>\ell}S_k\gamma^{2k}$ has light (Poisson branching) tails, but
its conditional mean $\sum_{k>\ell}m_kD_k$ fluctuates on the scale
$\kappa^{\ell/2}$ with only the polynomial tail control that $L^2$
bounds on the Kesten--Stigum martingale provide; upgrading Chebyshev to
a conditional sub-Gaussian bound therefore requires exponential control
of $(D_k)$, which we do not have. The floor $16/(\kappa-1)$ in
\cref{thm:above} is certainly not tight (as stated, it is
non-vacuous only for $\kappa>17$, hence only for average degree above
$17$; \cref{tab:certify} spells out the consequence at the
simulated parameters). One route to a floor of the right order: the
bound charges the bad branching event $E_W^c$ with error $1$, although
the classifier retains the root's own feature evidence there, so
multiplying $\Prob(E_W^c)$ by a non-trivial conditional error---which
requires left-tail control of the Kesten--Stigum limit $W_\infty$ near
$0$---should shrink the floor toward the universal
$e^{-\Delta}\varepsilon_{\mathrm{feat}}$ of
\cref{prop:floor}. Small-value and harmonic-moment estimates
for $W_\infty$ are classical for supercritical Galton--Watson
processes, so the program is plausible; what is missing is a version
with explicit constants at moderate degree, and we have not pursued
it. Above the threshold two basic questions
remain open: the \emph{true} rate at which $\err(\ell)$ approaches its
large-$\ell$ behaviour (our geometric term is a free-parameter
truncation cost, and the signal-driven terms are super-geometric), and
the existence of $\lim_{\ell\to\infty}\err(\ell)$ itself, which our
Cauchy argument establishes only below the threshold while the empirical
drift of \cref{obs:nonmono} shows the sequence is not
eventually monotone. On \emph{lower} bounds, the ledger after
\cref{thm:fliplower} and the certified instance of
\cref{app:certified} reads: decision churn at rate $\kappa^{\ell/2}$ is proven, and
non-monotonicity is proven in an exact instance; what remains open is a
lower bound on the \emph{error} scale---a statement of the form
$|\err(\ell)-\err(\infty)|\ge c\,\kappa^{C\ell}$ for some instance,
which would make the $O(\log(1/\epsilon))$ prescription two-sided.
Flips do not imply it: the two directions of decision change nearly
cancel in $\err(\ell)$, so the error difference can in principle decay
faster than the churn, and the simulations cannot distinguish these
possibilities at accessible depths.

\paragraph{Exact belief propagation.} \cref{rem:bp} isolates where
the pairwise rule and exact BP part ways, and the data-processing
inequality guarantees BP's depth curve is monotone; the BP baseline of
\cref{sec:experiments} exhibits exactly this, pulling ahead of
the pairwise rule by
$0.02$--$0.03$ error at depth $6$. On rates, the flip measurements of
\cref{fig:flips} \emph{revise} our earlier expectation: BP's
decisions settle strictly faster than the pairwise rule's, in
quantitative agreement with the annealed effective ratio
$\kappa_{\mathrm{BP}}=\Delta\,\E[F_\gamma'(L)^2]<\kappa$. We accordingly
conjecture: the dichotomy for BP is still organized by $\kappa$ (the
threshold is $\kappa=1$, as in robust reconstruction
\citep{janson2004robust}), but its below-threshold saturation exponent
is governed by $\kappa_{\mathrm{BP}}$, the pairwise $\kappa^{\ell/2}$
being only an upper envelope, approached as the feature evidence
vanishes ($\kappa_{\mathrm{BP}}\uparrow\kappa$ as
$\vartheta\downarrow0$). Proving this requires controlling the
nonlinear recursion under the stationary message law---the
information-contraction arguments of \citet{evans2000broadcasting} give
the $\kappa$ envelope; sharpening them to the stationary derivative law
is open, and we leave it to future work.

\paragraph{Extensions.} The heterophilous case $b>a$ follows by replacing
$\gamma$ with $-\gamma$: messages then alternate sign with distance parity
and $\kappa=\gamma^2\Delta$ is unchanged, so the dichotomy is untouched---a
statistical argument for signed, parity-aware aggregation on heterophilous
graphs. Natural next steps include: multi-class models, where $\gamma$ is
replaced by the second eigenvalue $\lambda$ of the class-transition matrix
and $\kappa=\lambda^2\Delta$; the semi-supervised setting, where revealed
labels act as clamped messages and interact with the threshold in the
spirit of \citet{kanade2016global}; and unbalanced classes. The last of these is not routine: the engine of
\cref{thm:below} is the exact cancellation supplied by
\cref{ass:sym}, and under class asymmetry the mean message of
generation $k$ acquires a label-\emph{independent} drift
$\sum_k S_k(m_k^++m_k^-)/2$-type term that $\sgn(T_\ell)$ does not
re-center; whether saturation survives for a suitably recentred
statistic, or the below-threshold picture is genuinely brittle to
asymmetry, is open.

\section*{Acknowledgements}
This manuscript develops a question left open in the author's doctoral
thesis \citep{baranwal2024thesis}; the framework of asymptotic local
classification on sparse feature-decorated graphs is joint work with Kimon
Fountoulakis and Aukosh Jagannath, whose influence on this line of inquiry
is gratefully acknowledged.

\bibliographystyle{plainnat}
\bibliography{references}

\appendix
\crefalias{section}{appendix}

\section{Complements to the below-threshold analysis}
\label{app:complements}

The proof of \cref{thm:below} gives, verbatim, the following
form of the saturation bound, which does not use
\cref{ass:density}:
\[
\big|\err(\ell)-\err(\ell')\big|\;\le\;\inf_{s>0}\;
\Big\{\,\mathcal{Q}_{T_{\ell'}}(s)\;+\;\frac{C_1\kappa^{\ell+1}}{s^2}\Big\},
\qquad
\mathcal{Q}_X(w):=\sup_{a\in\R}\Prob\big(X\in(a,a+w]\big),
\]
with \cref{ass:density} entering only to evaluate the L\'evy
concentration function as $\mathcal{Q}_{T_{\ell'}}(s)\le 2Bs$ through
the root's unbounded evidence. For atomic (e.g.\ discrete) feature laws
the form with $\mathcal{Q}$ is the correct statement of
\cref{thm:below}.

\subsection{Proof of Theorem~\ref{thm:fliplower}}

\begin{proof}
Condition on $y_o=+1$ (\cref{lem:symmetry}) and write
$T_{\ell+1}=T_\ell+R$ with $R=\sum_{v\in N_{\ell+1}}M_{\ell+1}(X_v)$.
Conditionally on $\mathcal{G}$, the variables $T_\ell$ and $R$ are
independent (they involve disjoint sets of features). Fix
$w:=\kappa^{(\ell+1)/2}\le1$ and observe the inclusion
\[
\big\{T_\ell\in(0,w]\big\}\cap\big\{R\le-2w\big\}
\;\subseteq\;\big\{h_\ell=+1,\;h_{\ell+1}=-1\big\}
\;\subseteq\;\big\{h_\ell\neq h_{\ell+1}\big\},
\]
since on the left event $T_{\ell+1}\le w-2w<0$. We lower-bound the
probability of the left event by conditioning on a good
$\mathcal{G}$-measurable event.

\emph{Step 1: the good event.} Let
$p_2:=\tfrac{\Delta-1}{4\Delta}$, $V_0:=\tfrac{8C_1\kappa}{p_2}$,
$M:=\sqrt{\tfrac{8}{p_2(1-\kappa)}}$, and set
\[
G_1=\Big\{\E\big[(T_\ell-\log\psi(X_o))^2\,\big|\,\mathcal{G}\big]\le
V_0\Big\},\qquad
G_2=\Big\{S_{\ell+1}\ge\tfrac12\Delta^{\ell+1}\Big\},
\]
\[
G_3=\Big\{|D_{\ell+1}|\le M(\gamma\Delta)^{\ell+1}\kappa^{-(\ell+1)/2}\Big\}.
\]
By \cref{lem:tail} applied with $(\ell,\ell')=(0,\ell)$ and Markov's
inequality, $\Prob(G_1^c)\le C_1\kappa/V_0=p_2/8$. Writing
$\widetilde{W}_j:=S_j/\Delta^j$, the shape martingale satisfies
$\E\widetilde{W}_j=1$ and
$\E\widetilde{W}_j^2\le1+\tfrac1{\Delta-1}$ (the computation of
\cref{lem:ks}(iii) with $\gamma=1$), so the Paley--Zygmund
inequality gives
$\Prob(G_2)\ge\tfrac{(1/2)^2}{\E\widetilde W_j^2}
\ge\tfrac{\Delta-1}{4\Delta}=p_2$, uniformly in $\ell$. By
\cref{lem:ks}(iv), $\E D_{\ell+1}^2\le
(\gamma\Delta)^{2(\ell+1)}\kappa^{-(\ell+1)}/(1-\kappa)$, so
$\Prob(G_3^c)\le\tfrac{1}{M^2(1-\kappa)}=p_2/8$. Hence
$G:=G_1\cap G_2\cap G_3$ has $\Prob(G)\ge3p_2/4$.

\emph{Step 2: the root passes near zero.} Conditionally on
$\mathcal{G}$, $T_\ell=\log\psi(X_o)+U$ with $\log\psi(X_o)$ independent
of $(\mathcal{G},U)$ and, on $G_1$, $\E[U^2|\mathcal{G}]\le V_0$, so
$\Prob(|U|\le\sqrt{2V_0}\,|\,\mathcal{G})\ge\tfrac12$. Since
$\sqrt{2V_0}+w\le 8\sqrt{C_1\kappa\Delta/(\Delta-1)}+1=K_0$, on
$\{|U|\le\sqrt{2V_0}\}$ the interval $(-U,-U+w]$ lies in $[-K_0,K_0]$,
where the density of $\log\psi(X_o)$ is at least $\beta$; hence on $G_1$,
\begin{equation}\label{eq:root-window}
\Prob\big(T_\ell\in(0,w]\,\big|\,\mathcal{G}\big)\;\ge\;\tfrac12\,\beta w.
\end{equation}

\emph{Step 3: the new generation swings low.} Conditionally on
$\mathcal{G}$, $R$ is a sum of $S_{\ell+1}$ independent terms, each of
range $\le2c_\gamma\gamma^{\ell+1}$, with conditional mean
$\mu_R=m_{\ell+1}D_{\ell+1}$ and conditional variance
$\sigma_R^2=\sum_{v\in N_{\ell+1}}\Var(M_{\ell+1}(X_v)\,|\,y_v)$. Since
$\artanh(u)=u+u^3\theta/(3(1-u^2))$ with $\theta\in[0,1]$, we have
$M_k(x)=2\gamma^k\big(t(x)+e_k(x)\big)$ with
$|e_k|\le\gamma^{2k}/(3(1-\gamma^2))=:\delta_k$, so for
$k\ge k_1:=\min\{k:\delta_k\le\sigma_t/2\}$,
$\Var(M_k(X)\,|\,y)\ge4\gamma^{2k}(\sigma_t-\delta_k)^2\ge
\gamma^{2k}\sigma_t^2$ under either class (the two classes agree by
\cref{ass:sym}). Hence on $G_2$, for $\ell+1\ge k_1$,
\[
\sigma_R^2\;\ge\;\tfrac12\Delta^{\ell+1}\cdot\gamma^{2(\ell+1)}\sigma_t^2
=\tfrac{\sigma_t^2}{2}\,\kappa^{\ell+1},
\qquad\text{while on }G_3\quad
|\mu_R|\le c_\gamma\sqrt{\vartheta}\,M\,\kappa^{(\ell+1)/2}
\]
using $m_{\ell+1}\le c_\gamma\gamma^{\ell+1}\sqrt{\vartheta}$
(\cref{lem:messages}(iii)). Therefore
$(2w+|\mu_R|)/\sigma_R\le z_0:=
\sqrt{2}\,\big(2+c_\gamma\sqrt{\vartheta}M\big)/\sigma_t$, and the
Berry--Esseen inequality for sums of independent, non-identically
distributed bounded variables \citep[Ch.~V]{petrov1995limit} gives, with
an absolute constant $C_0$,
\[
\Prob\big(R\le-2w\,\big|\,\mathcal{G}\big)
\;\ge\;\Phi(-z_0)\;-\;C_0\,
\frac{2c_\gamma\gamma^{\ell+1}}{\sigma_R}
\;\ge\;\Phi(-z_0)-\frac{2\sqrt2\,C_0c_\gamma}{\sigma_t}\,
\Delta^{-(\ell+1)/2},
\]
where we used $\sum_v\E|M_v-\E M_v|^3\le
2c_\gamma\gamma^{\ell+1}\sigma_R^2$ and
$\gamma^{\ell+1}/\kappa^{(\ell+1)/2}=\Delta^{-(\ell+1)/2}$. Let
$\ell_0\ge k_1$ be the smallest depth from which the Berry--Esseen
correction is at most $\Phi(-z_0)/2$; then
$\Prob(R\le-2w\,|\,\mathcal{G})\ge\Phi(-z_0)/2=:p_3$ on $G_2\cap G_3$
for $\ell\ge\ell_0$.

\emph{Step 4: assemble.} By conditional independence, for
$\ell\ge\ell_0$,
\[
\Prob\big(h_\ell\neq h_{\ell+1}\big)
\;\ge\;\E\Big[\mathbf{1}_G\,
\Prob\big(T_\ell\in(0,w]\,\big|\,\mathcal{G}\big)\,
\Prob\big(R\le-2w\,\big|\,\mathcal{G}\big)\Big]
\;\ge\;\frac{3p_2}{4}\cdot\frac{\beta w}{2}\cdot p_3
\;=\;c\,\kappa^{\ell/2},
\]
with $c=\tfrac38\,p_2\,\beta\,p_3\,\sqrt{\kappa}$. For the Gaussian
mixture, $\sigma_t^2=\Var\big(\tanh(Z/2)\big)>0$ for
$Z\sim\mathcal{N}(2\zeta^2,4\zeta^2)$ and the density of
$\log\psi(X)$ is bounded below on $[-K_0,K_0]$ by
$\beta=\varphi\big((K_0+2\zeta^2)/(2\zeta)\big)/(2\zeta)$, with
$\varphi$ the standard normal density.
\end{proof}

\section{The first layer: proof and refinements}\label{app:first}

\subsection{Proof of Proposition~\ref{prop:first}}

\begin{proof}
Write the depth-$1$ observation as $(X_o,\mathcal{B})$ with
$\mathcal{B}=\big(S_1,(X_v)_{v\in N_1}\big)$, the children listed in
exchangeable order; conditionally on $y_o$, $\mathcal{B}$ is independent
of $X_o$. Let $Q_\pm$ denote the conditional law of $\mathcal{B}$ given
$y_o=\pm1$. Three structural facts drive the proof.

\emph{(a) $h_1$ is Bayes for $(X_o,\mathcal{B})$.} By Poisson thinning,
$S_1\sim\Poi(\Delta)$ has the same law under both classes, and given
$y_o$ the children's features are i.i.d.\ with law
$R_{\pm}=\tfrac{1\pm\gamma}{2}\mathbb{P}_{+}+\tfrac{1\mp\gamma}{2}
\mathbb{P}_{-}$ under $y_o=\pm1$. The likelihood ratio of a single child
feature is
$\frac{dR_+}{dR_-}(x)=\frac{(1+\gamma)\rho_++(1-\gamma)\rho_-}
{(1-\gamma)\rho_++(1+\gamma)\rho_-}(x)
=\frac{1+\gamma t(x)}{1-\gamma t(x)}$,
whose logarithm is $2\artanh(\gamma t(x))=M_1(x)$. By conditional
independence, the posterior log-likelihood ratio of $y_o$ given
$(X_o,\mathcal{B})$ is exactly
$\log\psi(X_o)+\sum_{v\in N_1}M_1(X_v)=T_1$, so $h_1=\sgn(T_1)$ is a
Bayes rule for the depth-$1$ observation; since ties are null,
$\err(1)$ \emph{is} the corresponding Bayes error, and likewise
$\err(0)=\E[\min(p,1-p)]$ for $p:=\Prob(y_o=+1\,|\,X_o)$.

\emph{(b) Conditional Bayes decomposition.} Fix $X_o$ and let $q_\pm$ be
densities of $Q_\pm$ with respect to a common base measure $\nu$ (which
exist: $S_1$ is shared and $R_+,R_-$ are mutually absolutely continuous by
\cref{ass:sym}). With $f_p:=p\,q_+-(1-p)\,q_-$, so that $\int
f_p\dd\nu=2p-1$, the identity $\int|f_p|=|{\textstyle\int}f_p| +
2\int (f_p)_{\mp}$ (with the sign opposite to that of $2p-1$) gives
\[
\E\big[\text{error of Bayes rule}\,\big|\,X_o\big]
=\int \min\big(p\,q_+,(1-p)\,q_-\big)\dd\nu
=\min(p,1-p)\;-\;\Xi(p),
\]
where $\Xi(p):=\int(f_p)_-\dd\nu\ge0$ for $p\ge\tfrac12$ and
$\int(f_p)_+\dd\nu$ for $p<\tfrac12$. Taking expectations,
$\err(0)-\err(1)=\E[\Xi(p)]\ge0$, which is the monotonicity claim.

\emph{(c) $\Xi$ is large when the root is uncertain.} Suppose
$|2p-1|\le\delta$ and $p\ge\tfrac12$ (the other case is symmetric). Then
$p\le\tfrac{1+\delta}{2}$ and $1-p\ge\tfrac{1-\delta}{2}$, so pointwise
$-f_p\ \ge\ \tfrac12(q_--q_+)-\tfrac{\delta}{2}(q_++q_-)$, and using
$(u-v)_+\ge u_+-v$ for $v\ge0$ and $\int(q_++q_-)\dd\nu=2$,
\[
\Xi(p)\;\ge\;\tfrac12\int (q_--q_+)_+\dd\nu \;-\;\delta
\;=\;\tfrac12\,\mathrm{TV}(Q_+,Q_-)\;-\;\delta.
\]
Since $S_1$ is observed and label-independent,
$\mathrm{TV}(Q_+,Q_-)=\sum_{s\ge0}\Prob(S_1=s)\,
\mathrm{TV}\big(R_+^{\otimes s},R_-^{\otimes s}\big)
\ge\Prob(S_1\ge1)\,\mathrm{TV}(R_+,R_-)$, by restricting to the first
coordinate; and $R_+-R_-=\gamma\,(\mathbb{P}_+-\mathbb{P}_-)$ gives
$\mathrm{TV}(R_+,R_-)=\gamma\,\tau_{\mathrm{feat}}$ exactly. Hence
$\mathrm{TV}(Q_+,Q_-)\ge\tau_1$.

Finally, $2p-1=\tanh\big(\tfrac12\log\psi(X_o)\big)$, so
$\{|2p-1|\le\delta\}=\{|\log\psi(X_o)|\le\log\tfrac{1+\delta}{1-\delta}\}$,
and combining (b) and (c),
\[
\err(0)-\err(1)\;\ge\;
\E\Big[\Xi(p)\,\mathbf{1}\{|2p-1|\le\delta\}\Big]
\;\ge\;\Big(\frac{\tau_1}{2}-\delta\Big)\,
\Prob\Big(|\log\psi(X_o)|\le\log\tfrac{1+\delta}{1-\delta}\Big).
\qedhere
\]
\end{proof}

\begin{remark}[many children, and where the slack lies]\label{rem:tensor}
The definition of $\tau_1$ keeps a single child. Tensorizing instead
through the Hellinger distance---$\mathrm{TV}\ge H^2/2$ applied to
$R_+^{\otimes s}$ versus $R_-^{\otimes s}$, the product rule
$1-\tfrac12H^2(R_+^{\otimes s},R_-^{\otimes s})
=(1-\tfrac12H^2)^{s}$, and the Poisson generating function over
$S_1$---gives the closed form
\[
\mathrm{TV}(Q_+,Q_-)\;\ge\;1-e^{-\Delta H^2/2}\;\ge\;1-e^{-\kappa\vartheta/2},
\qquad\text{where}\quad
H^2:=H^2(R_+,R_-)\;\ge\;\gamma^2\vartheta,
\]
with $H^2=2\big(1-\E_{\mathrm{mix}}\sqrt{1-\gamma^2t(X)^2}\,\big)$,
$\E_{\mathrm{mix}}$ the expectation under the balanced feature
mixture, and $R_\pm$ the single-child laws from the proof above: the Kesten--Stigum ratio reappears as the
first layer's information budget. Either bound may dominate: the tensorized one tends to $1$ as
$\Delta\to\infty$ while the one-child bound stays below
$\gamma\tau_{\mathrm{feat}}$, but at the simulated parameters of
\cref{sec:experiments} the one-child bound is larger ($0.200$
versus $0.088$ below threshold, $0.372$ versus $0.294$ above). Neither
is the main source of slack in \eqref{eq:first-layer}: a Monte Carlo
evaluation of the \emph{true} total variation of the depth-$1$ data
($0.32$ below, $0.57$ above, at $\zeta=0.6$) raises the certified
improvement only to ${\approx}\,0.007$ and ${\approx}\,0.023$
respectively, still well short of the measured $0.044$ and $0.118$. The
dominant slack is the restriction to the near-tie event: the proof
credits the layer only where the root is nearly undecided, whereas the
true improvement $\err(0)-\err(1)=\E[\Xi(p)]$ (in the notation of the
proof) accrues at every posterior $p$.
\end{remark}

\section{A certified instance of non-monotonicity}\label{app:certified}

\begin{proposition}[non-monotonicity certified: an exact instance]
\label{prop:nonmonocert}
Consider the two-point feature law $t(X)\in\{\pm t_0\}$ with
$\mathbb{P}_+\big(t(X)=t_0\big)=\tfrac{1+t_0}{2}$ (which satisfies
\crefrange{ass:sym}{ass:informative}), at
$(\Delta,\gamma,t_0)=(3,\,0.55,\,0.95)$, so that $\kappa=0.9075<1$. Then
\[
\err(0)=0.025,\qquad
\err(1)=0.02261938\ldots,
\]
\[
\err(2)\in[0.03870641,\,0.03870642],\qquad
\err(3)\in[0.07589622,\,0.07589623],
\]
and in particular
\[
\err(3)-\err(2)\;\ge\;0.0371898\;>\;0
\qquad\text{and}\qquad
\err(2)-\err(1)\;\ge\;0.0160\;>\;0:
\]
the depth-error curve of the pairwise rule strictly increases from depth
$1$ to depth $3$. The proof is the finite, certified computation described below; the
two-point law has atoms, so this
instance also illustrates the discrete-feature regime outside
\cref{ass:density}.
\end{proposition}

Under the two-point law, the message of a vertex at distance $k$ is
$c_k s_v$ with $c_k=2\artanh(\gamma^kt_0)$ and $s_v\in\{\pm1\}$,
$\Prob(s_v=+1\,|\,\sigma_v)=\tfrac{1+t_0\sigma_v}{2}$; hence
$T_\ell=\sum_{k\le\ell}c_km_k$ with $m_k=\sum_{v\in N_k}s_v\in\Z$, and
$\err(2),\err(3)$ are determined by the joint law of
$(s_0,m_1,m_2,m_3)$ on a lattice. That law is a finite object up to
Poisson truncation, by repeated use of thinning:
\begin{enumerate}[leftmargin=2em,itemsep=0.1em]
\item the level-$3$ signs attached to one level-$2$ vertex of relative
label $\sigma_2$ sum to a $\mathrm{Skellam}\big(3q,3(1-q)\big)$ variable,
$q=\tfrac{1+\gamma t_0\sigma_2}{2}$;
\item a group of $n$ level-$2$ vertices of common label contributes an
independent pair: the level-$2$ signs sum to
$2\,\mathrm{Binom}(n,p_2)-n$ with $p_2=\tfrac{1+t_0\sigma_2}{2}$, and
their level-$3$ children to $\mathrm{Skellam}(3nq,3n(1-q))$;
\item each level-$1$ vertex's children split into independent
$\Poi\big(\tfrac{\Delta(1+\gamma)}{2}\big)$ same-label and
$\Poi\big(\tfrac{\Delta(1-\gamma)}{2}\big)$ flipped-label groups;
convolving the two groups gives the per-child law
$A(m_2,m_3\,|\,\sigma_1)$, with $A(\cdot\,|\,{-}1)$ the reflection of
$A(\cdot\,|\,{+}1)$ by global sign symmetry;
\item conditioning on the root's group sizes $(i,j)$ makes $m_1$ an
independent sum of two shifted binomials and $(m_2,m_3)$ the convolution
$A(\cdot|{+}1)^{*i}*A(\cdot|{-}1)^{*j}$.
\end{enumerate}
The computation evaluates these finite convolutions on truncated integer
boxes with a strict accounting rule: every distribution is kept as a
\emph{sub}-probability and never renormalized, so each array's mass
deficit is exactly the truncation error it carries; the Poisson weights
beyond the group-size caps ($i\le16$, $j\le9$, level-$2$ groups
$\le20$/$\le12$) are added to the error ledger as-is; lattice points
with $|T_\ell|<10^{-8}$ are counted as errors (the tie-conservative
side); and a conservative floating-point budget covers rounding. The
resulting certified brackets are those displayed in
\cref{prop:nonmonocert}, with total ledger below
$5\times10^{-9}$---six orders of magnitude below the certified gap. Two
independent checks: the reflection symmetry of step 3 is verified by
building both laws separately ($L^1$ discrepancy $5\times10^{-16}$), and
a $4\times10^5$-tree Monte Carlo run reproduces all four values within
two standard errors. The script (\texttt{code/nonmono\_certify.py})
runs in minutes.

\end{document}